\documentclass[a4paper]{article}[12pts]

\usepackage[english]{babel}
\usepackage[utf8x]{inputenc}
\usepackage[T1]{fontenc}

\normalsize

\usepackage[a4paper,top=1in,bottom=1in,left=1in,right=1in,marginparwidth=1in]{geometry}

\usepackage{setspace}
\usepackage{amsmath}
\usepackage{amssymb} 
\usepackage{amsthm}
\usepackage{amsfonts, mathtools}
\usepackage{algorithm,algpseudocode}
\usepackage{bm}
\usepackage{bbm}
\usepackage{comment}
\usepackage{graphicx}
\usepackage{multirow, booktabs}
\usepackage{caption}
\usepackage{subcaption}
\usepackage[colorinlistoftodos]{todonotes}
\usepackage[colorlinks=true, allcolors=blue]{hyperref}
\usepackage{thmtools}
\usepackage{thm-restate}
\usepackage{hyperref}
\usepackage{cleveref}

\usepackage{siunitx} 

\newcommand{\E}{\mathbb{E}}

\newtheorem{assumption}{Assumption}
\newtheorem{theorem}{Theorem}
\newtheorem{thm}{Theorem}[section]
\newtheorem{corollary}{Corollary}

   \newtheoremstyle{TheoremNum}
        {\topsep}{\topsep}              
        {\itshape}                      
        {}                              
        {\bfseries}                     
        {.}                             
        { }                             
        {\thmname{#1}\thmnote{ \bfseries #3}}
    \theoremstyle{TheoremNum}

   \newtheoremstyle{TheoremNum}
        {\topsep}{\topsep}              
        {\itshape}                      
        {}                              
        {\bfseries}                     
        {.}                             
        { }                             
        {\thmname{#1}\thmnote{ \bfseries #3}}
    \theoremstyle{TheoremNum}

\usepackage{natbib}

\definecolor{rose}{rgb}{1.0, 0.33, 0.64}

\title{Out-of-distribution Robust Optimization}
\author{Zhongze Cai$^{\dagger}$, Hansheng Jiang$^{\ddagger}$, Xiaocheng Li$^{\dagger}$}
\date{\small
$^{\dagger}$Imperial College Business School, Imperial College London, (z.cai22, xiaocheng.li)@imperial.ac.uk  \\ 
$^{\ddagger}$Rotman School of Management, University of Toronto, hansheng.jiang@rotman.utoronto.ca
}

\begin{document}
\maketitle

\begin{abstract}
In this paper, we consider the contextual robust optimization problem under an out-of-distribution setting. The contextual robust optimization problem considers a risk-sensitive objective function for an optimization problem with the presence of a context vector (also known as covariates or side information) capturing related information. While the existing works mainly consider the in-distribution setting, and the resultant robustness achieved is in an out-of-sample sense, our paper studies an out-of-distribution setting where there can be a difference between the test environment and the training environment where the data are collected. We propose methods that handle this out-of-distribution setting, and the key relies on a density ratio estimation for the distribution shift. We show that additional structures such as covariate shift and label shift are not only helpful in defending distribution shift but also necessary in avoiding non-trivial solutions compared to other principled methods such as distributionally robust optimization. We also illustrate how the covariates can be useful in this procedure. Numerical experiments generate more intuitions and demonstrate that the proposed methods can help avoid over-conservative solutions.
\end{abstract}

\section{Introduction}

Contextual optimization considers a constrained optimization problem with the presence of covariates (context), and it can be viewed as a prediction problem under an optimization context where the output of the prediction model serves as the objective function for the downstream optimization problem. The goal is to develop a model (trained from past data) that prescribes a decision/solution for the downstream optimization problem using the covariates directly but without observation of the objective function. It has been extensively studied in recent years under various settings \citep{hu2022fast, elmachtoub2022smart, bertsimas2020predictive, ho2022risk, chen2020online, wilder2019melding, sun2023maximum, huang2024learning}, and we refer to \cite{sadana2024survey} for a survey on the topic. In this paper, we consider the problem of contextual robust optimization where the objective function of the optimization problem becomes a risk-sensitive one. Unlike all existing works, we consider a distribution shift/out-of-distribution setting where the test environment is different from the training environment that generates the training dataset. In this sense, the notion of robustness goes beyond the existing scope and it also covers robustness against the distribution shift. To summarize, our contributions are as follows:

First, we formulate the problem of out-of-distribution robust optimization and propose a method that utilizes a density ratio estimate to make inferences about the test environment with data from the training environment. 
 
Second, we derive theoretical guarantees for the proposed method and use an analytical example to illustrate the values of shift structure in avoiding over-conservative solutions.

Third, we conduct numerical experiments to generate more intuitions for the setup and illustrate the effectiveness of the proposed method.


\textbf{Related Literature.} The study of robust optimization has a long history which considers risk-sensitive objectives and approximate solutions using the notion of uncertainty set (see \cite{ghaoui2003worst, chen2007robust,natarajan2008incorporating,huang2023robust,qiao2023topology} among others). A recent line of works considers the contextual formulation for robust optimization \cite{chenreddy2022data, sun2024predictthencalibratenewperspectiverobust, patel2024conformal, chenreddy2024endtoendconditionalrobustoptimization} which can be viewed as the case of no distribution shift within our framework.  A separate line of works concerns the problem of distributionally robust optimization for conditional or contextual stochastic optimization problems, and various techniques have been proposed, including \cite{esteban2022distributionally,kannan2020residuals, bertsimas2022bootstrap, liu2022distributionally}. Treatment of distributional shift in this context is very scarce except for recent work by \citep{wang2024contextual}, which proposed to treat covariate shift by combining two estimators. \cite{duchi2023distributionally} considered distribution shift in the form of the latent mixture under a DRO framework, but their focus is on loss functions rather than optimization. As for the literature on distribution shift, there are several streams: (i) Classification (label shift): \cite{lipton2018detectingcorrectinglabelshift} propose to correct for label shift in any predictors, but their method only applies to the classification setting, i.e., a finite number of labels; (ii) Regression (label shift): \cite{zhang2013domain} handle continuous labels with the kernel mean matching approach; (iii) Covariate shift: Reweighting is a fundamental idea to address covariate shift, dating back to \cite{shimodaira2000improving}. Covariate shift is also considered in the causal inference literature, typically caused by unobserved confounders \cite{jin2023model}.




\section{Problem Setup}
\subsection{Robust Contextual LP}

\label{subsec:contextualLP}
Consider a standard-form linear program (LP)
\begin{align}
    \label{lp:std}
    \text{LP}(c,A,b) \coloneqq \min_x \ &  c^\top x,\\
    \text{s.t.\ } &  Ax=b, \ x\ge 0, \nonumber
\end{align}
where $c\in\mathbb{R}^n,$ $A\in\mathbb{R}^{m\times n},$ and $b\in\mathbb{R}^m$ are the inputs of the LP, and $x\in\mathbb{R}^n$ is the decision variable. A more data-driven setting of LP considers the presence of some covariates (contextual information or side information), where there is a feature vector $z\in\mathbb{R}^d$ that contains some information related to the objective vector $c$. In this way, the LP can be described by the tuple $(c,A,b,z)$.


The statistical setup of \textit{contextual optimization} considers the tuple $(c,A,b,z)$ drawn from some (unknown) distribution $\mathcal{P}$. As the setup of a machine learning problem, there is an available (training) dataset consisting of i.i.d. samples from $\mathcal{P}$,
$$\mathcal{D}=\{(c_i, A_i, b_i, z_i)\}_{i=1}^N.$$
One utilizes the dataset to develop a model so that in the test phase, one needs to recommend a feasible solution $x_{\text{new}}$ to a new LP problem using only the observation of  $(A_{\text{new}},b_{\text{new}},z_{\text{new}})$ but without observing the objective vector $c_{\text{new}}$. Here $(c_{\text{new}}, A_{\text{new}},b_{\text{new}},z_{\text{new}})$ is a new tuple independently sampled from the distribution $\mathcal{P}.$ In terms of the objective, contextual optimization usually adopts a risk-neutral objective and aims to optimize (in the test phase)
\begin{align}
\label{lp:condi_exp}
 \min_x \ &  \E[c|z]^\top x,\\
    \text{s.t.\ } &  Ax=b, \ x\ge 0 \nonumber,
\end{align}
where the conditional distribution $\E[c|z]$ is not known exactly (due to the unknown $\mathcal{P}$) and has to be learned from the training data. The recommended decision variables $x$ in the test phase can be viewed as a function of $(A,b,z)$.


Alternatively, one can consider a risk-sensitive (robust) objective 
\begin{align}
\label{lp:condi_var}
 \min_x \ &  \text{VaR}_{\alpha}(c^\top x|z),\\
    \text{s.t.\ } &  Ax=b, \ x\ge 0 \nonumber,
\end{align}
where $\alpha\in(0,1)$ is a pre-specified constant. Here $\text{VaR}_{\alpha}(U)$ denotes the $\alpha$-quantile/value-at-risk of a random variable $U$. Specifically, $\text{VaR}_{\alpha}(U)\coloneqq F^{-1}_{U}(\alpha)$ with $F^{-1}_{U}(\cdot)$ being the inverse cumulative distribution function of $U$.
For the objective \eqref{lp:condi_exp}, it concerns the conditional expectation of $\E[c|z]$, while \eqref{lp:condi_var} involves the quantile of the conditional distribution $c^\top x|z.$

From the literature of robust optimization \citep{ben2009robust}, we know that the problem \eqref{lp:condi_var} can be equivalently written as the following optimization problem 
\begin{align}
\label{opt:intractable}
 \min_{x,\mathcal{U}} \max_{c\in \mathcal{U}} \ &  c^\top x,\\
    \text{s.t.\ } &  Ax=b, \ x\ge 0, \ \mathbb{P}_{c|z}(c\in \mathcal{U})\geq \alpha, \nonumber
\end{align}
where the decision variables become $x$, $c$, and the uncertainty set $\mathcal{U}$. The last constraint ensures the coverage guarantee and corresponds to the quantile level $\alpha$; the probability is taken with respect to the conditional distribution $c|z.$ Due to the intractability of \eqref{opt:intractable}, people usually consider the following problem as an approximation
\begin{align}
\label{opt:tractable}
\text{LP}(\mathcal{U}) \coloneqq \min_{x} \max_{c\in \mathcal{U}} \ &  c^\top x,\\
    \text{s.t.\ } &  Ax=b, \ x\ge 0, \nonumber
\end{align}
where $\mathcal{U}$, instead of being a decision variable, is a fixed uncertainty set satisfying $\mathbb{P}_{c|z}(c\in \mathcal{U})\geq \alpha$. The uncertainty set ideally covers the high-density region so that the approximation to \eqref{lp:condi_var} is tighter.

\subsection{Distribution Shift}\label{subsec:distshift_intro}

For both the risk-neutral \eqref{lp:condi_exp} and the risk-sensitive  \eqref{lp:condi_var} settings, the existing works on contextual optimization and robust context optimization mainly consider an in-distribution setting where the training data $\mathcal{D}$ is sampled from some unknown distribution $\mathcal{P}$, and the distribution $\mathcal{P}$ remains unchanged from training to test. In this spirit, the so-called robustness of the existing methods is achieved against either (i) the statistical gap between the realized samples $\mathcal{D}$ and the unknown $\mathcal{P}$ or (ii) the intrinsic uncertainty underlying the distribution $\mathcal{P}$ that causes the randomness of $c|z.$ In this paper, we re-examine this classic problem of robust optimization from the perspective of robustness against \textit{distribution shift} or an \textit{out-of-distribution} setting. While such a robustness guarantee can only be achieved in a very conservative manner in the worst case, we demonstrate (i) how the covariates can be useful and (ii) how additional structure on the distribution shift can reduce the conservativeness. 
 
In the last subsection, we define a joint distribution on the tuple $(c,A,b,z)$. Throughout the remainder of this paper, we omit describing the distribution of the constraint $(A,b)$ and focus on the joint distribution of the objective-covariates pair $(c,z)$. This is without loss of generality because the constraint $(A,b)$ is revealed when solving the LP during the test phase. In general, a distribution shift or an out-of-distribution setting refers to the case in which the test data distribution differs from the training data distribution. Specifically, we consider a training data 
$$\mathcal{D}_{\text{Tr}}\coloneqq \{(c_i,z_i)\}_{i=1}^N\sim \mathcal{P}$$
sampled from some unknown distribution $\mathcal{P}.$ Different from the in-distribution case, the sample of the test phase 
$$(c_{\text{new}},z_{\text{new}})\sim\mathcal{Q},$$
where $\mathcal{Q}$ may be different from $\mathcal{P}.$ The standard setup of contextual optimization (for both risk-neutral objective \eqref{lp:condi_exp} and risk-sensitive objective \eqref{lp:condi_var}) considers the case of $\mathcal{P}=\mathcal{Q}$, whereas an out-of-distribution robust optimization allows a setting where $\mathcal{P}\neq \mathcal{Q}.$

In context-free robust optimization, if one adopts the approximation formulation \eqref{opt:tractable}, the task reduces to characterizing the randomness of $c$ from the observed samples $\{c_1,...,c_N\}$. For contextual robust optimization, the task accordingly becomes characterizing the randomness of the conditional distribution $c|z.$ When it comes to the out-of-distribution setting, it uses the samples from the training distribution $\mathcal{P}$ to obtain a characterization of the conditional distribution $c|z$ under the test distribution $\mathcal{Q}.$ Hence, people usually assume an additional dataset is available
$$\mathcal{D}_{\text{Te}} \coloneqq \{z_1',...,z_M'\}\sim \mathcal{Q},$$
where it contains $M$ samples from $\mathcal{Q}$ but only has the covariates $z_{i}'$ but no corresponding $c_{i}'.$ An alternative setting includes also $c_{i}'$ in the dataset $\mathcal{D}_{\text{Te}}$. We consider this weaker setting for the full generality.

It is imaginable and will be illustrated in the later sections that such distribution shift from $\mathcal{P}$ and $\mathcal{Q}$ cannot be resolved in a worst-case sense. Specifically, if $\mathcal{Q}$ is allowed to be arbitrarily different from $\mathcal{P}$, there is no way we can learn anything meaningful from dataset $\mathcal{D}_{\text{Tr}}$. In this sense, we consider two common structures as in the distribution shift literature: \textit{covariate shift} and \textit{label shift}. 

Notation-wise, let $\mathcal{P}_z, \mathcal{P}_c$ denote the marginal distribution of $z$ and $c$, and $\mathcal{P}_{z|c}, \mathcal{P}_{c|z}$ denote the conditional distribution of $z|c$ and $c|z$. In this notation, $\mathcal{P}=\mathcal{P}_c\times \mathcal{P}_{z|c} =\mathcal{P}_z\times \mathcal{P}_{c|z}.$ We define $\mathcal{Q}_z, \mathcal{Q}_c, \mathcal{Q}_{z|c}$ and $\mathcal{Q}_{c|z}$ analogously. Throughout this paper, we assume all the distributions have density functions. Let $p(z)$ and $p(c)$ denote the marginal density of $z$ and $c$ under distribution $\mathcal{P}$, and $p_{c|z}(c;z)$ denote the conditional density of $c|z$. The corresponding densities for $q$ are defined analogously.

\begin{itemize}
    \item Covariate shift: The distribution of the covariates $z$ differs between $\mathcal{P}$ and $\mathcal{Q}$, while the conditional distribution $c|z$ remains the same. In other words, $\mathcal{P}_z\neq \mathcal{Q}_z$ but $\mathcal{P}_{c|z} = \mathcal{Q}_{c|z}$. 
    \item Label shift: The distribution of the objective vector $c$ differs between $\mathcal{P}$ and $\mathcal{Q}$, while the conditional distribution $z|c$ remains the same. In other words, $\mathcal{P}_c\neq \mathcal{Q}_c$ but $\mathcal{P}_{z|c} = \mathcal{Q}_{z|c}$.
\end{itemize}
In the following, we first state our algorithm in general and then specialize it into these two cases.

\section{OOD Robust Optimization}

Following the setup in the previous section, the goal of robust optimization under the test distribution $\mathcal{Q}$ is 
\begin{align*}
 \min_x \ &  \text{VaR}_{\alpha}^{\mathcal{Q}}(c^\top x|z),\\
    \text{s.t.\ } &  Ax=b, \ x\ge 0 \nonumber,
\end{align*}
where the quantile function VaR is with respect to the conditional distribution of $c^\top x|z$ under $\mathcal{Q}$. If one adopts the approximation scheme, it reduces to solving the following problem 
\begin{align}
\label{opt:dist_shift_OPT}
\text{LP}(\mathcal{U}) \coloneqq \min_{x} \max_{c\in \mathcal{U}} \ &  c^\top x,\\
\text{s.t.\ } &  Ax=b, \ x\ge 0, \nonumber
\end{align}
where the uncertainty set $\mathcal{U}$ satisfies $\mathbb{P}^{\mathcal{Q}}_{c|z}(c\in \mathcal{U})\geq \alpha$, i.e., $\mathcal{U}$ gives a coverage guarantee for the conditional distribution $c|z$ under $\mathcal{Q}.$ Compared to the in-distribution problem \eqref{opt:tractable}, the formulation \eqref{opt:dist_shift_OPT} replaces the requirement on the uncertainty set $\mathcal{U}$ from $\mathcal{P}$ to $\mathcal{Q}.$ Then the task reduces to constructing an uncertainty set $\mathcal{U}$ for the distribution $\mathcal{Q}$ using the data $\mathcal{D}_{\text{Tr}}$ and $\mathcal{D}_{\text{Te}}.$ Consequently, suppose we have a perfect knowledge of the distribution $\mathcal{P}$, then it only needs an estimate of the density ratio to convert our knowledge of $\mathcal{P}$ into an estimate of the test distribution $\mathcal{Q}$.

We formally define the density ratio between $\mathcal{P}$ and $\mathcal{Q}$ as
$$w(c,z) \coloneqq \frac{q(c,z)}{p(c,z)}$$
for all $(c,z)$ where $q$ and $p$ are the density functions for $\mathcal{Q}$ and $\mathcal{P}$, respectively. We assume the density ratio is well-defined everywhere. In particular, for the case of covariate shift, the density ratio concerns only the marginal distribution over $z$, i.e., $w(c,z) = \frac{q(z)}{p(z)}$. Generally, we need to estimate the ratio function $w$ from the data $\mathcal{D}_{\text{Tr}}$ and $\mathcal{D}_{\text{Te}}.$ The following algorithm takes an estimate as its input, and later in Section \ref{sec:DR_estimate}, we show how it can be estimated under the two settings of covariate shift and label shift.

\subsection{Algorithm}

Algorithm \ref{alg:BUQ-OOD} describes our main algorithm for OOD robust optimization.  It takes the training data $\mathcal{D}_{\text{Tr}}$ and the target quantile level $\alpha$ as inputs. Also, it requires a prediction model $\hat{f}$ for the conditional expectation $\E[c|z]$ and a density ratio estimate $\hat{w}(c,z)$. The prediction model $\hat{f}$ can be built based on some other available history data or using part of the training data $\mathcal{D}_{\text{Tr}}$. The predicted vector from this $\hat{f}$ determines the center of the uncertainty set $\mathcal{U}.$

The algorithm consists of two parts. The first part gives a coarse prediction of the quantiles of the conditional distribution $c|z.$ It does not have to be accurate or even consistent. In the second part, we use both the expectation prediction model $\hat{f}$ and the quantile prediction model $\hat{h}$ to construct a variable-sized uncertainty set. The uncertainty set is box-shaped and it is controlled by a scalar variable $\eta.$ The parameter \(\eta\) plays the role of scaling the uncertainty bounds predicted by the quantile regression model \( \hat{h}(z) \). Intuitively, a larger value of \(\eta\) yields a more conservative uncertainty set, enhancing robustness at the cost of increased conservatism. Conversely, a smaller \(\eta\) produces tighter uncertainty bounds, potentially improving decision efficiency but risking inadequate coverage.

Step \ref{step:scaleCaliber1} in Algorithm \ref{alg:BUQ-OOD} is crucial for determining the size of the uncertainty set. In this step, it utilizes the estimated density ratio $\hat{w}$ to re-weight the training samples in $\mathcal{D}_2$. The re-weighting adjusts the original \textit{uniform} weight over all the training samples in $\mathcal{D}_2$ and assigns a weight to each sample to reflect its density in the test distribution $\mathcal{Q}$. Intuitively, step \ref{scaleCaliber1} ensures an empirical coverage of the samples over the test distribution pretending the density ratio estimate is precise.

\begin{algorithm}[ht!]
    \caption{OOD-RO with Box Uncertainty Set}
    \label{alg:BUQ-OOD}
    \begin{algorithmic}[1] 
    \State \textbf{Input}: Training data $\mathcal{D}_{\mathrm{Tr}}$, a prediction model $\hat{f}$ (for $\E[c|z]$), a density ratio estimate $\hat{w}(c,z)$, target quantile level $\alpha$ \nonumber
    \State Initialization: Randomly split the training set into two sets such that $\mathcal{D}_{\mathrm{Tr}} = \mathcal{D}_{1} \cup \mathcal{D}_{2}$ and $\mathcal{D}_{1} \cap \mathcal{D}_{2} = \emptyset$
    \State \textcolor{blue}{\%\% \textit{Base quantile prediction}}
    \For{$(c_i,z_i) \in \mathcal{D}_1$} \label{step:quant_pred}
    \State Calculate the residual vector on the $i$-th training sample
    \begin{align}
        r_i \coloneqq c_i - \hat{f}(z_i)
    \end{align}
    \quad \,\,and denote $r_i=(r_{i,1},...,r_{i,n})^\top$
    \EndFor
    \State Learn a quantile regression model $\hat{h}(z): \mathbb{R}^d \rightarrow \mathbb{R}^n$ by minimizing
    $$
    \sum_{(c_i,z_i)\in\mathcal{D}_1} \sum_{k=1}^n \rho_{\alpha}\left(\hat{h}(z_i)_k-|r_{i,k}|\right)
    $$
    where $\rho_{\alpha}(\cdot) \coloneqq \alpha\ (\cdot)^+ + (1-\alpha)\  (\cdot)^-$ denotes the pinball loss
\State \textcolor{blue}{\%\% \textit{Confidence adjustment}}
\For{$(c_i,z_i)\in \mathcal{D}_2$} \label{step:conformal_pred}
\State Let
$$
\begin{aligned}
    &\bar{c}_{i}(\eta) \coloneqq \hat{f}(z_i) +  \eta \hat{h}(z_i)\  \in \mathbb{R}^n,\\
    &\underline{c}_{i}(\eta) \coloneqq \hat{f}(z_i)- \eta\hat{h}(z_i) \ \in \mathbb{R}^n 
\end{aligned}
$$
\EndFor
\State Choose a minimal $\eta>0$ such that
\begin{equation}
\begin{aligned}
    \sum_{(c_i,z_i)\in \mathcal{D}_2} \dfrac{\hat{w}(c_i,z_i)\cdot \mathbbm{1}\{\underline{c}_{i}(\eta)\le c_i\le \bar{c}_{i}(\eta)\}}{\sum_{(z_j, c_j)\in \mathcal{D}_2} \hat{w}(c_j,z_j)}  \ge \alpha
\end{aligned}
\label{scaleCaliber1}
\end{equation}
where $\mathbbm{1}\{\cdot\}$ is the indicator function and  the inequality within the indicator function is required to hold component-wise \label{step:scaleCaliber1}
\State \textbf{Output}: $\hat{h}$, $\eta$. For any $z\in \mathcal{Z}$, produce 
$$\mathcal{U}_{\alpha}(z) = \left[\hat{f}(z) -\eta\hat{h}(z), \hat{f}(z) +\eta\hat{h}(z)\right]
$$
\end{algorithmic}
\end{algorithm}

We make a few remarks on the algorithm. First, the output uncertainty set $\mathcal{U}_{\alpha}(z)$ from the algorithm is contextualized and depends on the covariates $z$, with the benefits of contextual uncertainty set are illustrated in the existing literature \citep{goerigk2020data, chenreddy2022data, sun2024predictthencalibratenewperspectiverobust}. Second, we note that the algorithm utilizes the idea from the recent literature on conformal prediction to adjust the size of the uncertainty set. Such usage of conformal prediction in robust optimization and inverse optimization have appeared in several recent works \citep{sun2024predictthencalibratenewperspectiverobust, patel2024conformal, lin2024conformal, cao2024non}. The main advantage of using this conformal control is to obtain a coverage guarantee for the uncertainty set without imposing any assumptions/structures on the prediction models $\hat{f}$ and $\hat{h}.$ Third, our re-weighting scheme \eqref{scaleCaliber1} is a natural treatment and it mimics the existing literature on uncertainty calibration and conformal prediction \citep{tibshirani2020conformalpredictioncovariateshift,podkopaev2021distribution}. Moreover, we note that the inputs of the algorithm require only the training data $\mathcal{D}_{\text{Tr}}$ but not the test data $\mathcal{D}_{\text{Te}}.$ The test data contributes only to the construction of the density ratio estimate $\hat{w}$, while Algorithm \ref{alg:BUQ-OOD} does not directly involve the test data. Lastly, we note the algorithm provides just a pipeline to solve the out-of-distribution robust optimization; many of its components, say, the density ratio estimate, the prediction model $\hat{f}$, and the learning of the quantile, are unrestricted in choice and can be substituted with other candidate methods.  

\subsection{Algorithm Analysis}

Now we derive a guarantee for Algorithm \ref{alg:BUQ-OOD} under the following assumption.

\begin{assumption}\label{assump:u_l_bound}
We assume that the estimated density ratio is uniformly bounded from below and above. That is, there exists $\underline{w}>0$ and $\bar{w}>0$ such that
    $$
    \underline{w}\leq \hat{w}(c,z)\leq \bar{w}
    $$
    for all $(c,z)$.
\end{assumption}

We note this assumption is rather mild in that it concerns the estimated density ratio $\hat{w}(c,z)$ but not the true density ratio $w(c,z)=q(c,z)/p(c,z)$. Such a condition can always be met by truncating the estimated density ratio. Boundedness of the true density ratio is also a common assumption adopted in the literature \citep{kpotufe2017lipschitz,ma2023optimally}.


\begin{theorem}\label{thm:cvg_alpha} Under Assumption \ref{assump:u_l_bound}, suppose the density ratio estimate is perfect, i.e., $\hat{w}(c,z)=w(c,z)=q(c,z)/p(c,z)$, then the uncertainty set $\mathcal{U}_{\alpha}(z)$ generated by Algorithm \ref{alg:BUQ-OOD} satisfies the following coverage guarantee,
    \begin{equation*}
        \Big| \mathbb{P}\left(c_{\mathrm{new}} \in \mathcal{U}_\alpha(z_{\mathrm{new}})\right) -\alpha\Big|\leq \dfrac{1}{|\mathcal{D}_2|+1}\cdot \dfrac{\bar{w}}{\underline{w}}
    \end{equation*}
    where the probability on the left-hand-side is with respect to $(c_{\text{new}},z_{\text{new}})\sim \mathcal{Q}$ and $\mathcal{D}_2 \sim \mathcal{P}$.
\end{theorem}

Theorem \ref{thm:cvg_alpha} considers the case when the density ratio estimate is precise. Under this condition, the performance guarantee holds and the right-hand side scales at a favorable rate as the number of samples in $\mathcal{D}_2$ of Algorithm \ref{alg:BUQ-OOD} increases. We emphasize that the algorithm utilizes samples from the training distribution $\mathcal{P}$ to construct the uncertainty set but the performance guarantee is with respect to the test distribution $\mathcal{Q}.$ This tells that the out-of-distribution robustness is well achievable under a perfect knowledge of density ratio. 

When the density ratio estimate is not perfect, one can also obtain a slightly weaker guarantee. To proceed, let $\hat{\mathcal{Q}}$ denote the estimated test distribution defined by the following density function 
\begin{equation}\label{eq:approx_density}
    \hat{q}(c,z) = \dfrac{\hat{w}(c,z)\cdot p(c,z)}{\int_{(c,z)}\hat{w}(c,z)\cdot p(c,z) \mathrm{d}z\mathrm{d}c}.
\end{equation}
With this definition, the following corollary extends Theorem \ref{thm:cvg_alpha} to the case when the density ratio estimate is not fully accurate. 


\begin{corollary}
    Let $\hat{\mathcal{Q}}$ denote a distribution with a density function given by \eqref{eq:approx_density}, then  under Assumption \ref{assump:u_l_bound}, the uncertainty set $\mathcal{U}_\alpha(z)$ generated by Algorithm \ref{alg:BUQ-OOD} satisfies the following guarantee,
    \begin{equation*}
        \Big|\mathbb{P}\left(c_{\mathrm{new}} \in \mathcal{U}_\alpha(z_{\mathrm{new}})\right) -\alpha\Big| \leq \dfrac{1}{|\mathcal{D}_2|+1}\cdot \dfrac{\bar{w}}{\underline{w}} + D_{\mathrm{TV}}(\mathcal{Q}, \hat{\mathcal{Q}}),
    \end{equation*}
     where the probability on the left-hand-side is with respect to $(c_{\mathrm{new}},z_{\mathrm{new}})\sim \mathcal{Q}$ and $\mathcal{D}_2 \sim \mathcal{P}$. Here $ D_{\mathrm{TV}}(\cdot, \cdot)$ denotes the total variation distance between two distributions.
     \label{coro:guarantee}
\end{corollary}

Corollary \ref{coro:guarantee} provides a natural extension that includes an extra term in the bound of Theorem \ref{thm:cvg_alpha}. When the density ratio estimate is perfect, $\hat{\mathcal{Q}}=\mathcal{Q},$ and this extra term disappears. At the other end of the spectrum, if one adopts a trivial density estimate where $\hat{w}(c,z)\equiv 1,$ this term becomes $D_{\mathrm{TV}}(\mathcal{Q}, \mathcal{P})$ which is the distance between the training distribution and the test distribution. In this light, the bound shows the benefits of having a good density ratio estimate; and any reasonable density estimate will always bring a better performance guarantee than the naive treatment of $\hat{w}(c,z)\equiv 1$ which ignores the distribution shift between $\mathcal{P}$ and $\mathcal{Q}.$ The two terms on the right-hand-side correspond to two ``orthogonal'' sources of errors: the first one captures the uncertainty calibration error when there is no distribution shift, and the second one captures the error of density ratio estimation. 

We note that for the performance guarantees in both Theorem \ref{thm:cvg_alpha} and Corollary \ref{coro:guarantee}, they are in a marginal sense, not a conditional sense. In other words, they hold ``on average'' for all possible covariates $z$. To obtain a conditional coverage guarantee that holds for each possible $z$ requires additional assumptions and treatments even for the in-distribution setting (see \citep{kannan2022data, sun2024predictthencalibratenewperspectiverobust} among others). It will be accordingly more challenging for an out-of-distribution setting, which we leave for future investigation. Numerically, as the construction of the quantile prediction (in part 1 of Algorithm \ref{alg:BUQ-OOD}) is in a conditional sense, we usually observe a conditional coverage empirically.

\subsection{Density Ratio Estimation for Covariate Shift and Label Shift}
\label{sec:DR_estimate}

Algorithm \ref{alg:BUQ-OOD} takes a density ratio estimate as input. Here we describe how such estimate can be obtained under covariate shift and label shift.

\textbf{Covariate Shift}. We briefly describe the probabilistic classification approach proposed by \cite{bickel2009discriminative}. Intuitively, consider a random sample $z_i$ drawn from a mixture distribution of $\mathcal{P}$ and $\mathcal{Q}$. With equal probability, $z_i$ is either drawn from $\mathcal{P}$ or from $\mathcal{Q}$. If $z_i$ is drawn from $\mathcal{P}$, it is assigned a label $l_i=0$, and if it is drawn from $\mathcal{Q}$, it is assigned a label $l_i=1$. Note that the following equation holds
$$
\dfrac{\mathbb{P}(l_i=1|z_i)}{\mathbb{P}(l_i=0|z_i)} = \dfrac{q(z_i)}{p(z_i)}.
$$
This indicates that if we train a probability classifier that predicts the label $l_i$ based on $z_i$, then the oracle classifier gives $\mathbb{P}(l_i=1|z_i) = q(z_i)/ (q(z_i) + p(z_i))$. 

We can obtain a density ratio estimate based on this observation. First, we pool the covariates $z_i$ from both $\mathcal{D}_{\text{Te}}$ and $\mathcal{D}_{\text{Te}}$. We assign a label $l_i=0$ if $z_i$ is from $\mathcal{D}_{\text{Tr}}$ and $l_i=1$ if $z_i$ is from $\mathcal{D}_{\text{Te}}$. Then, we learn a probability classifier $\hat{p}$ (can be an ML model) so that $\hat{p}(z_i)\approx \mathbb{P}(l_i=1|z_i)$. The density ratio can then be given by
$
\hat{w}(c_i,z_i) = \dfrac{\hat{p}(z_i)}{1-\hat{p}(z_i)}.
$

\textbf{Label Shift}. Most of the existing literature on density ratio estimation under the label shift setting focuses on the classification task where the domain $\mathcal{C}$ for the target variable $c$ is a discrete space. However, the cost vectors $c_i$'s in contextual optimization problems are continuous vectors. We can adopt the kernel mean matching approach introduced in \cite{zhang2013domain}. A brief introduction to the reproductive kernel Hilbert space (RKHS) and kernel mean match (KMM) is deferred to Appendix \ref{appdix:RKHS}. Let $(\mathcal{F}, \kappa, \mathcal{Z})$ and $(\mathcal{G}, \eta, \mathcal{C})$ be RKHS's, with $\phi[z]$ and $\psi[c]$ denoting the respective  feature maps. The kernel mean embedding operators $\mu_{\mathcal{F}}[\cdot], \mu_{\mathcal{G}}[\cdot]$ are defined as
$$
\mu_{\mathcal{F}}[\mathcal{P}_z]\coloneqq \mathbb{E}_{Z\sim \mathcal{P}_z}[\phi[Z]], \ \ 
\mu_{\mathcal{G}}[\mathcal{P}_c]\coloneqq \mathbb{E}_{C\sim \mathcal{P}_c}[\psi[C]].
$$
Define the conditional embedding operator as $\mathcal{U}_{Z|C}\coloneqq \mathcal{A}_{Z,C} \mathcal{A}_{C,C}^{-1}$, where $\mathcal{A}_{Z,C}$ and $\mathcal{A}_{C,C}$ denote the cross-covariance operators \citep{fukumizu2004dimensionality}. Based on the property that 
$$
\mu_{\mathcal{F}}[\mathcal{P}_z] = \mathcal{U}_{Z|C}[\mu_{\mathcal{G}}[\mathcal{P}_c]],
$$
the following equation holds under the label shift assumption:
$
\mu_{\mathcal{F}}[\mathcal{Q}_z] = \mathcal{U}_{Z|C}[\mu_{\mathcal{G}}[\mathcal{Q}_c]].
$
The main idea then is to estimate the density ratio with $\hat{w}(c,z)$ such that the estimated test distribution $\hat{\mathcal{Q}}_c$ satisfies $\mu_{\mathcal{F}}[\mathcal{Q}_z] \approx \mathcal{U}_{Z|C}[\mu_{\mathcal{G}}[\hat{\mathcal{Q}}_c]]$. Following the empirical estimators given in Appendix \ref{appdix:empirical_estimate}, the $L_2$ loss to minimize can be written as
$$
\begin{aligned}
    \left\| \hat{\mathcal{U}}_{Z|C}\left(\dfrac{1}{N}\sum_{i=1}^N w(c_i,z_i)\cdot \psi[c_i]\right)  - \dfrac{1}{M} \sum_{j=1}^M \phi[z_j'] \right\|^2
\end{aligned}.
$$
The estimate $\hat{w}(c, z)$ can then be chosen from a function family that minimizes the above $L_2$ loss.

\textbf{Algorithm Complexity.} In Appendix~\ref{appendix:complexity}, we provide a detailed analysis of the computational complexity of Algorithm~\ref{alg:BUQ-OOD}. As noted there, the density ratio estimation step is the primary factor limiting the algorithm’s scalability, particularly under the label shift setting, where the complexity is at least $\mathcal{O}(N^2)$ due to matrix inversion. Dimensionality reduction techniques can effectively mitigate this issue in practice.


\section{Values of Shift Structure}


In this section, we use an analytical example to show that the structure of covariates shift and label shift is helpful in avoiding over-conservative solutions. Consider the following linear program. 
\begin{align}
    \min_x \ & \ \text{VaR}_{\alpha}^{\mathcal{Q}}(c\cdot x|z)  \ \ \text{s.t.}\ -1\le x\le 1, \label{LP:toy}
\end{align}
where $\alpha\in (0.5,1)$. For the training distribution $\mathcal{P}$, $c = z + \epsilon$ with $z$ and $\epsilon$ independently sampled from $\mathcal{N}(0, \sigma_1^2)$ and $\mathcal{N}(0, \sigma_2^2)$. Under this example, the ideal solution would be $x^* = -\text{sign}(z)$ if the coverage guarantee can be met. An over-conservative solution will give $x^*=0$. In the following, we analyze the problem under covariate shift and label shift.  

\paragraph{Covariate Shift.} We introduce a scalar $s>0$ to represent the extent of the distribution shift. Under the covariate shift setting, the test distribution $\mathcal{Q}$ of $z$ follows $\mathcal{N}(s, \sigma_1^2)$, while the distribution of $\epsilon$ remains unchanged. Hence the conditional distribution of $c|z$ remains to be $\mathcal{N}(z, \sigma_2^2)$. The joint distribution of $(c,z)$ is shifted from $
\mathcal{P} = \mathcal{N}\left(
\begin{pmatrix}
0\\
0
\end{pmatrix}
, \begin{pmatrix}
    \sigma_1^2+\sigma_2^2 & \sigma_1^2\\
    \sigma_1^2 & \sigma_1^2
\end{pmatrix}
\right)$ to $\mathcal{Q} = \mathcal{N}\left(
 \begin{pmatrix}
s\\
s
\end{pmatrix}
, \begin{pmatrix}
    \sigma_1^2+ \sigma_2^2 & \sigma_1^2\\
    \sigma_1^2 & \sigma_1^2 
\end{pmatrix}
\right)$. We consider an idealized setting where the density ratio estimation is perfect and the distribution $\mathcal{P}$ is known. Under this case, the solution of Algorithm \ref{alg:BUQ-OOD} to \eqref{LP:toy} takes the following form
\begin{equation}
    x^* = \left\{
\begin{aligned}
    &1, & \text{VaR}^{\mathcal{Q}}_\alpha(c|z) \leq 0,\\
    &-1, & \text{VaR}^{\mathcal{Q}}_{1-\alpha}(c|z) \geq 0,\\
    &0, & \text{otherwise},
\end{aligned}
\right.\label{eq:our_oracle_solution}
\end{equation}
where $\text{VaR}^{\mathcal{Q}}_\alpha(c|z)$ is with respect to the test distribution $\mathcal{Q}$. The probability of yielding a conservative solution can be calculated by
\begin{equation}\label{eq:toy_solution}
\begin{aligned}
    & \mathbb{P}(x^* =0) = \mathbb{P}_{z\sim \mathcal{Q}_z}(\text{VaR}_{1-\alpha}(c|z)\leq 0\leq \text{VaR}_\alpha(c|z)) \\
    &= \Phi\left(\frac{\sigma_2}{\sigma_1}\Phi^{-1}(\alpha)-\frac{s}{\sigma_1}\right) - \Phi\left(-\frac{\sigma_2}{\sigma_1}\Phi^{-1}(\alpha)-\frac{s}{\sigma_1}\right)
\end{aligned}
\end{equation}
where $\Phi(\cdot)$ is the cumulative distribution function of a standard normal distribution.

Alternatively, one can take a worst-case approach by considering the following set of distributions which we call as \textit{worst-case ball}: 
$$
\mathcal{B}(R)\coloneqq \left\{\mathcal{N}\left(
r
, \begin{pmatrix}
    \sigma_1^2  + \sigma_2^2 & \sigma_1^2\\
    \sigma_1^2 & \sigma_1^2
\end{pmatrix}
\right)\ : \ \|r\| \leq R\right\}
$$
where the set is centered at the training distribution $\mathcal{P}.$ Such a distribution set is commonly considered in distributionally robust optimization and we can accordingly formulate the robust optimization problem as 
\begin{align}
    \min_x \max_{\mathcal{Q}\in \mathcal{B}(R)} \ & \ \text{VaR}^{\mathcal{Q}}_{\alpha}(c\cdot x|z)  \ \ \text{s.t.}\ -1\le x\le 1. \label{LP:toy_DRO}
\end{align}
To make the set $\mathcal{B}(R)$ cover the test distribution $\mathcal{Q}$, we need to have $R=\sqrt{2}s$. Accordingly, we can solve \eqref{LP:toy_DRO} and obtain the following:
\begin{equation}\label{eq:toy_solution_DRO}
    x^* = \left\{
\begin{aligned}
    &1, & \overline{\text{VaR}}_\alpha(c|z) \leq 0,\\
    &-1, & \underline{\text{VaR}}_{1-\alpha}(c|z) \geq 0,\\
    &0, & \text{otherwise},
\end{aligned}
\right.
\end{equation}
where for any $\alpha\in (0,1)$ we define 
$$
\overline{\text{VaR}}_\alpha(c|z)\coloneqq \sup\{\text{VaR}_\alpha(c|z)\ : \ (c,z)\sim \mathcal{Q}\in \mathcal{B}(R)\},
$$ 
$$\underline{\text{VaR}}_{\alpha}(c|z)\coloneqq \inf\{\text{VaR}_{\alpha}(c|z)\ : \ (c,z)\sim \mathcal{Q}\in \mathcal{B}(R)\},
$$
and the probability of an overly conservative solution under this worst-case ball framework \eqref{LP:toy_DRO} equals
\begin{equation*}
\begin{aligned}
    &\mathbb{P}(x^* =0) = \mathbb{P}_{z\sim \mathcal{Q}_z}(\underline{\text{VaR}}_{1-\alpha}(c|z)\leq 0\leq \overline{\text{VaR}}_\alpha(c|z)) \\
    &=\Phi\left(\frac{\sigma_2}{\sigma_1}\Phi^{-1}(\alpha)+\frac{R}{\sigma_1}-\frac{s}{\sigma_1}\right) \\
    &\phantom{=} - \Phi\left(-\frac{\sigma_2}{\sigma_1}\Phi^{-1}(\alpha)-\frac{R}{\sigma_1}-\frac{s}{\sigma_1}\right).
\end{aligned}
\end{equation*}

\begin{figure}[ht!]
    \centering
    \begin{minipage}{0.45\columnwidth}
        \centering
        \includegraphics[width=\textwidth]{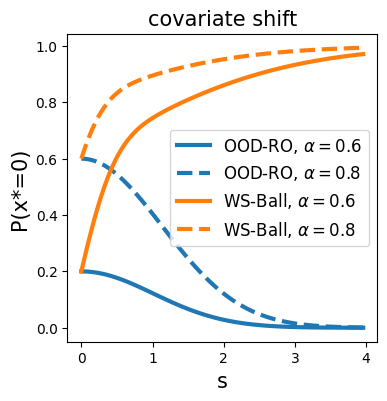} 
    \end{minipage}\hfill
    \begin{minipage}{0.45\columnwidth}
        \centering
        \includegraphics[width=\textwidth]{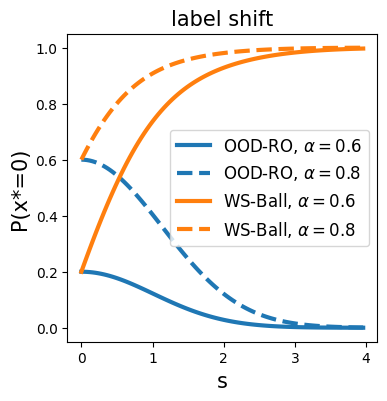} 
    \end{minipage}
    \caption{The curves show how $\mathbb{P}(x^*=0)$ changes with the distribution shift extent $s$. The label ``OOD-RO'' represents our approach and ``WS-Ball'' represents the worst-case approach. $\alpha$ denotes the risk level. As $s$ increases, our ``OOD-RO'' becomes less conservative, with $\mathbb{P}(x^*=0)$ decreasing to $0$, whereas the ``WS-Ball'' becomes increasingly conservative.}
\label{fig:toy_vs_OOD}
\end{figure}

We can do the same calculation for the case of label shift, where the details are deferred to Appendix \ref{sec:illustrative_example}. Figure \ref{fig:toy_vs_OOD} summarizes the two cases to compare our approach (which accounts for the shift structure) against the worst-case ball approach under two target quantile levels $\alpha=0.6$ and $0.8$. For both cases, as $s$ increases, the mean of the covariate $z$ and the objective $c$ becomes more and more positive. If we utilize this structure and also the information from the covariate, we should be able to take a more aggressive decision by setting $x*=-1$; this corresponds to the small probability of $\mathbb{P}(x^*=0)$ for our approach. However, if we do not take advantage of this shift structure, the worst-case ball has to be very large to cover the test distribution, and then a worst-case solution will result in a more and more conservative solution even as $s$ increases. This corresponds to the case in Corollary \ref{coro:guarantee} if one adopts the naive density ratio $\hat{w}\equiv 1$, the right-hand-side will have a term of $D_{\text{TV}}(\mathcal{Q}, \mathcal{P}).$


\section{Experiments}\label{sec:experiment}

In this section, we illustrate the performance of our proposed algorithm using a simple example. Additional experiments are provided in Appendix \ref{sec:Experi}. We do not include direct comparisons with existing baseline methods, as our out-of-distribution (OOD) robust optimization setting is fundamentally different. Traditional robust optimization and distributionally robust optimization (DRO) approaches typically assume that training and test distributions are identical and rely on predefined uncertainty sets that fail to capture the types of distribution shifts we consider. As such, these methods cannot be directly applied without significant modifications. Furthermore, standardized benchmarks or datasets tailored to our contextual OOD setting are not currently available. Thus, our evaluations primarily aim to demonstrate and analyze the internal effectiveness, robustness, and empirical behavior of our method under controlled, simulated distribution shifts.

Continue with problem \eqref{LP:toy}, but instead of considering a single-dimensional $z$, set a multi-dimensional covariate $z=(z_1,\ldots, z_d)$. Let 
$$
c = (\text{sign}(z_1) +\epsilon)\cdot \sqrt{|z_1|}.
$$
$z$ and $\epsilon$ are independent. The training distribution of $z$ and $\epsilon$ are $\mathcal{N}(\mathbf{0}, \mathbf{I}_d)$ and $\mathcal{N}(0, 0.1)$ respectively (let $\mathbf{0}$ denote a zero vector of dimension $d$, and $\mathbf{I}_d$ an identity matrix of dimension $d$). Covariate shift occurs in the test phase, where the mean of $z$ is shifted from $\mathbf{0}$ to $\mathbf{1}_d$ (let $\mathbf{1}_d$ denote a $d$-dimensional vector with all components equaling $1$),
$$
\mathcal{P}_z=\mathcal{N}(\mathbf{0}, \mathbf{I}_d)\longrightarrow \mathcal{Q}_z = \mathcal{N}(\mathbf{1}_d, \mathbf{I}_d).
$$
We study the performance of Algorithm \ref{alg:BUQ-OOD} under different implementations of $\hat{f}$, $\hat{h}$, and $\hat{w}$, as listed below. All neural networks used in this experiment are feed-forward networks with one hidden layer of $16$ neurons. We consider the following setups. Prediction models $\hat{f}$: Lasso, random forest (``RF''), and neural network (``NN''). Quantile prediction models $\hat{h}$: linear quantile regression (``Linear''), gradient boosting regression (``GBR''), and neural network. Density ratio estimators $\hat{w}$: the trivial estimator (``Trivial'', $\hat{w}\equiv 1$), the kernel mean matching method (``KMM'', \citep{gretton2008covariate}), and the probabilistic classification method (``Cls-NN'', \citep{bickel2009discriminative}).

The risk level $\alpha=0.8$. The training dataset contains $4000$ randomly generated samples, with $2000$ of them used for learning $\hat{f}$, $1000$ for learning $\hat{h}$, and $1000$ for confidence adjustment. The test data contains $5000$ samples from the shifted test distribution, with $4000$ of them used to learn the density ratio estimator, and $1000$ of them for evaluating the algorithm's final performance. Unless otherwise stated, the default covariate dimension is $d=4$. We present the main results of our experiment below, with further details and additional experimental results provided in Appendix \ref{sec:Experi}.

Prediction Models Boost Performance: 
In Figure \ref{fig:vary_f_h}, Algorithm \ref{alg:BUQ-OOD} is evaluated under different implementations of $\hat{f}$ and $\hat{h}$. The default configuration uses neural networks for $\hat{f}$ and $\hat{h}$, and probabilistic classification for $\hat{w}$. Except for the function component explicitly studied in the experiment, all other components follow the default configuration. Compared to the rest of the benchmarks, the neural network usually gives a better prediction, and Figure \ref{fig:vary_f_h} demonstrates that this reduces the probability of getting a conservative solution $x^*=0$. In Appendix \ref{sec:Experi}, we further show that neural networks give more precise predictions and that the probability of being conservative is positively correlated with the mean squared error of the prediction models.   

\begin{figure}[ht!]
    \centering
    \begin{minipage}{0.48\columnwidth}
        \centering
        \includegraphics[width=\textwidth]{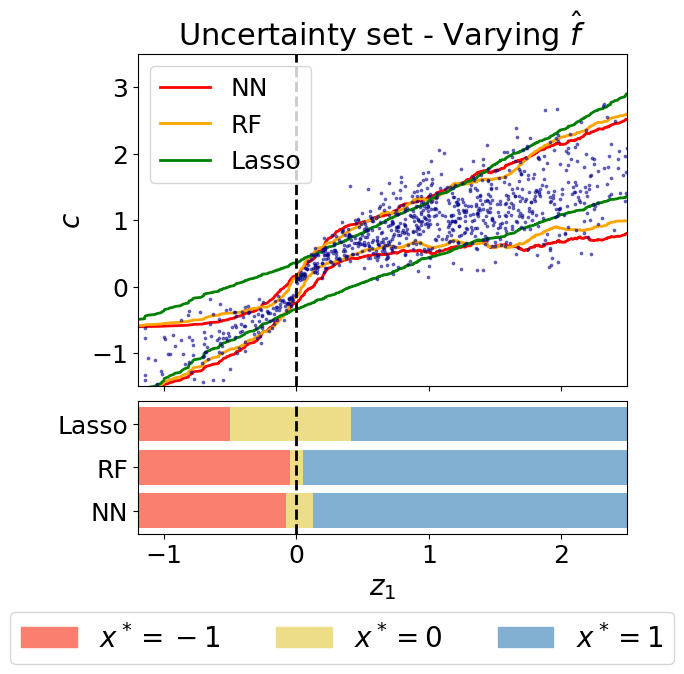} 
    \end{minipage}\hfill
    \begin{minipage}{0.48\columnwidth}
        \centering
        \includegraphics[width=\textwidth]{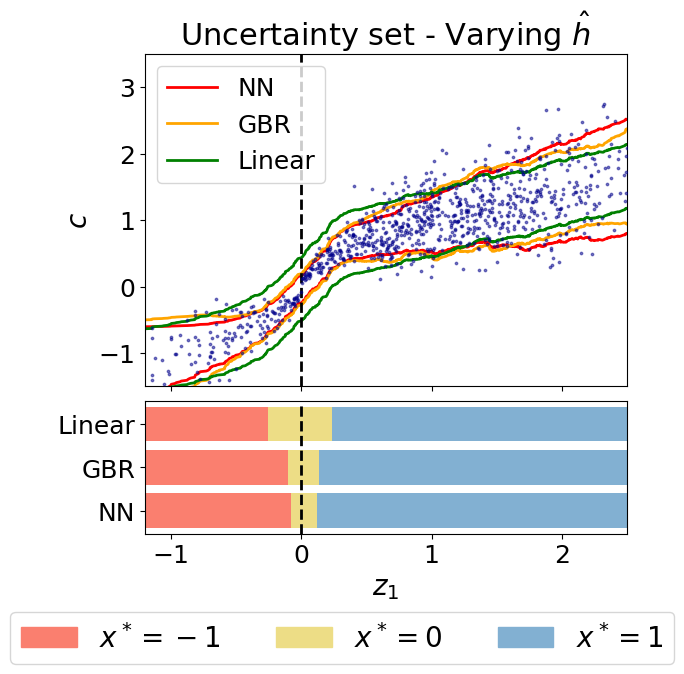} 
    \end{minipage}
    \caption{Performance of Algorithm \ref{alg:BUQ-OOD} under different prediction models $\hat{f}$ and quantile regression models $\hat{h}$. Each scattered point indicates a random test sample. The colored curves and the region between them indicate the (conditional) uncertainty sets. The left panel implements different $\hat{f}$ with $\hat{h}$ fixed to be a neural network, and the right panel implements different $\hat{h}$. The bar chart on the bottom indicates the corresponding optimal solution $x^*$ for $z_1$'s from different regions. Implementing a better prediction model (the ``NN'') generally reduces the probability of obtaining a conservative solution $x^*=0$.}
    \label{fig:vary_f_h}
\end{figure}

Density Ratio Estimator Avoids Mis-coverage:
Now we illustrate the role of the density ratio estimator $\hat{w}$. The left panel of Figure \ref{fig:figure_n_table} demonstrates the algorithm's performance under different $\hat{w}$. The figure indicates that, if no reweighting is applied (i.e. ``Trivially'' estimate the density ratio to be $1$), the empirical coverage rate can significantly drop below $0.8$. A potentially misleading observation from the bar chart is that a ``Trivial'' estimator appears to be less conservative than its counterparts. However, as discussed above, this deceptive advantage is achieved at the cost of violating the coverage guarantee. The table on the right panel of Figure \ref{fig:figure_n_table} reinforces the importance of a reliable density ratio estimator. By observing the ``Total'' rows of KMM and Cls-NN under different settings of the covariate dimension, we point out that a reliable density ratio estimator maintains the empirical coverage rate to be close to $0.8$, even when the prediction models are poor. Another observation is that the coverage guarantee is approximately achieved in a conditional sense: after splitting the domain of $z_1$ into two groups, with ``$z_1\leq 0$'' being a minority group and  ``$z_1>0$'' being major, the empirical coverage rate remains approximately $0.8$ for both groups.

\begin{figure}[h!]
    \centering
    \begin{minipage}{0.48\columnwidth}
        \centering
        \includegraphics[width=\textwidth]{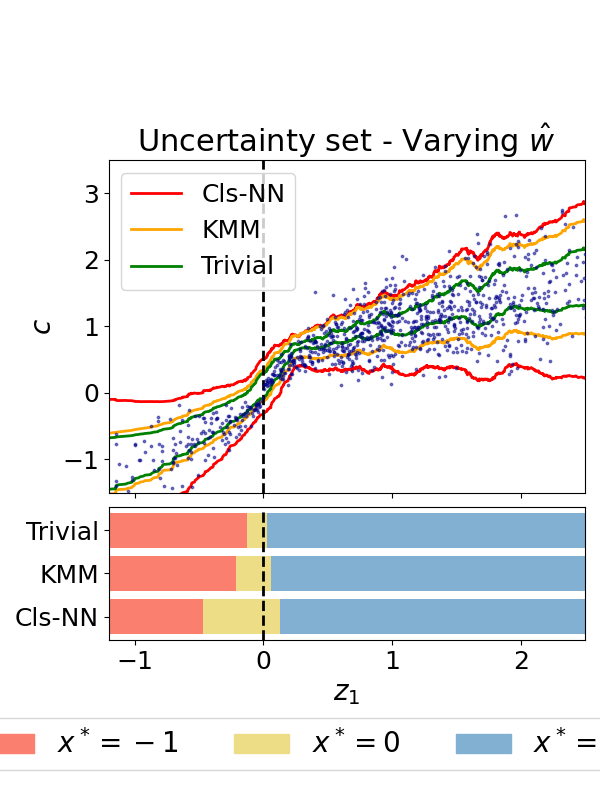} 
    \end{minipage}%
    \hfill
    \begin{minipage}{0.48\columnwidth}
        \centering
        \renewcommand{\arraystretch}{1.3}
        \resizebox{1. \columnwidth}{!}{
        \begin{tabular}{c c|c|c|c}
        \hline
        & & Trivial & KMM &  Cls-NN \\ \hline
        \multirow{3}{*}{$d=2$} & $z_1\leq 0$ & 0.80 & 0.77 &  0.79 \\ 
                               & $z_1> 0$   & 0.78 & 0.75 &  0.77 \\ 
                               & Total    & 0.78 & 0.75 &  0.78 \\ \hline
       \multirow{3}{*}{$d=4$} & $z_1\leq 0$ & 0.56 & 0.72 &  0.79\\ 
                               & $z_1> 0$   & 0.54 & 0.86 &  0.81\\ 
                               & Total    & 0.55 & 0.84  &  0.81 \\ \hline
        \multirow{3}{*}{$d=8$} & $z_1\leq 0$ & 0.53 & 0.71 & 0.86  \\ 
                               & $z_1> 0$   & 0.4 & 0.76 &  0.81 \\ 
                               & Total    & 0.42 & 0.75 &  0.83 \\ \hline
        \end{tabular}
        }
    \end{minipage}
    \caption{The figure on the left evaluates the performance of Algorithm \ref{alg:BUQ-OOD} under different density ratio estimators $\hat{w}$. With a better density ratio estimator (the ``NN''), the coverage rate of the uncertainty sets is closer to the target level of $0.8$. The table on the right gives the coverage rate of the uncertainty sets under different $\hat{w}$, different covariate dimensions $d$, and for different regions of $z_1$. For example, the row ``$z_1\leq 0$'' gives the coverage rate on samples with negative $z_1$, and the same holds for ``$z_1 > 0$''. The ``Total'' rows give the coverage rate on the whole test dataset.}
    \label{fig:figure_n_table}
\end{figure}

\textbf{Discussions.} In this paper, we study an out-of-distribution setting of robust optimization where the training distribution $\mathcal{P}$ might be different from the test distribution $\mathcal{Q}.$ Other than the method, the analysis, and the numerical experiments, we hope such a study calls for more attention to the meaning of robustness in the context of robust optimization. As noted earlier, existing works consider mainly (i) robustness against the statistical gap between the empirical distribution supported on $\mathcal{D}_{\text{Tr}}$ and the true distribution $\mathcal{P}$ or (ii) the robustness of a risk-sensitive objective such as VaR. Comparatively, the setting we study can be viewed as a robustness against distribution shift. Along this path, one may also consider the adversarial robustness where there might be a proportion of samples in the training data contaminated or adversarially generated from the other distributions than $\mathcal{P}$. We believe the study of such more general notions of robustness deserves more future works. 

\bibliographystyle{plainnat} 
\bibliography{main.bib}

\appendix

\newpage 

\section{Algorithm Complexity and Scalability}\label{appendix:complexity}

In this section, we analyze the computational complexity of Algorithm~\ref{alg:BUQ-OOD}, which primarily involves estimating three components: the conditional expectation estimator $\hat{f}$, the quantile regression model $\hat{h}$, and the density ratio estimator $\hat{w}$.
\begin{enumerate}
\item The conditional expectation estimation step typically employs methods such as LASSO regression, random forests, or neural networks. The computational complexity of this step generally scales linearly with the number of samples $N$ and the input dimension $d$.
\item The quantile regression model $\hat{h}$, trained using standard gradient-based methods, typically has a computational complexity of $\mathcal{O}(N d)$ per iteration.
\item The density ratio estimation step significantly affects the overall complexity. For covariate shift scenarios using logistic regression, the complexity is at least $\mathcal{O}(N d)$ per training iteration. Under label shift, the complexity of Kernel Mean Matching (KMM) is at least $\mathcal{O}(N^2)$ due to matrix inversion.
\end{enumerate}
Overall, the computational complexity of Algorithm~\ref{alg:BUQ-OOD} is primarily driven by the density ratio estimation step, particularly in label shift settings.

To address scalability challenges, practical solutions include dimensionality reduction via random projections or feature extraction techniques based on principal component analysis (PCA) or machine learning-based embeddings. Random projections, in particular, can substantially reduce the dimension $d$ to a lower-dimensional space $d'$, thereby reducing subsequent computational costs. Feature extraction methods further help preserve essential covariate information, offering a complexity saving without notable performance loss. Careful tuning of these dimensionality reduction techniques is recommended to achieve an optimal trade-off between computational efficiency and algorithmic performance.

\section{Reproductive Kernel Hilbert Space and Kernel Mean Matching}

\subsection{RKHS}\label{appdix:RKHS}
Let the tuple $(\mathcal{F}, \kappa, \mathcal{Z})$ denote a \textbf{reproductive kernel Hilbert space} (RKHS) $\mathcal{F}$ on the sample space $\mathcal{Z}$ with kernel $\kappa$. The RKHS is a Hilbert space of functions $f:\mathcal{Z}\to \mathbb{R}$, on which the inner product $\langle \cdot, \cdot \rangle_{\mathcal{F}}$ satisfies the reproducing property:
$$\langle f, \kappa(z, \cdot)\rangle_{\mathcal{F}} = f(z),\quad \forall f\in \mathcal{F}, z\in \mathcal{Z}.$$ 
That said, the evaluation of function $f$ on a single point $z$ can be viewed as an inner product between function $f$ and the evaluation operator $\kappa(z, \cdot)$. We define the \textbf{feature map} $\phi:\mathcal{Z}\to \mathcal{F}$ by $\phi[z]\coloneqq \kappa(z, \cdot)$. Throughout this part, we will use the square bracket $[\cdot]$ to denote mappings to functional spaces. Based on the reproducing property, the feature map satisfies $\langle \phi[z], \phi[z']\rangle_{\mathcal{F}} = \kappa(z, z')$ for all $z, z'\in \mathcal{Z}$.

The \textbf{kernel mean embedding} of a probability distribution $\mathcal{P}_z$ \citep{zhang2013domain}, denoted by $\mu_{\mathcal{F}}[\cdot]$, is a mapping from the space of all the probability distributions on $\mathcal{Z}$ to $\mathcal{F}$. The mapping is given by 
$$
\mu_{\mathcal{F}}[\mathcal{P}_z]\coloneqq \mathbb{E}_{Z\sim \mathcal{P}_z}[\phi[Z]].
$$
The kernel $\kappa$ is called \emph{characteristic} if the kernel mean embedding is injective: for $\mathcal{P}_z\neq \mathcal{P}_z'$, it holds that $\mu_{\mathcal{F}}[\mathcal{P}_z]\neq \mu_{\mathcal{F}}[\mathcal{P}_z']$. And following \cite{huang2006correcting}, the operator $\mu_{\mathcal{F}}[\cdot]$ is bijective if $\kappa$ is a universal kernel in the sense of \cite{steinwart}. The core idea of the \textbf{kernel mean matching} is estimate $\mathcal{P}_z$ with a $\hat{\mathcal{P}}_z$ satisfying $\mu_{\mathcal{F}}[\mathcal{P}_z]\approx \mu_{\mathcal{F}}[\hat{\mathcal{P}}_z]$.

\subsection{Cross-covariance Operator}
Consider the joint random variable $(Z,C)\in \mathcal{Z}\times \mathcal{C}$. Define two RKHS's by $(\mathcal{F}, \kappa, \mathcal{Z})$ and $(\mathcal{G}, \eta, \mathcal{C})$ respectively. Let $\mathcal{P}$ denote the joint distribution of $(Z, C)$ and $\mathcal{P}_z, \mathcal{P}_c$ their marginal distributions respectively, then the \textbf{cross-covariance operator} $\mathcal{A}_{Z,C}$ (with the dependency on $\mathcal{P}$ omitted) is defined as \citep{CrossCovariance}
$$
\mathcal{A}_{Z,C}\coloneqq \mathbb{E}_{(Z,C)\sim \mathcal{P}}\left[\phi[Z]\otimes \psi[C]\right] - \mathbb{E}_{Z\sim \mathcal{P}_z}[\phi[Z]]\otimes \mathbb{E}_{C\sim \mathcal{P}_c}[\psi[C]].
$$
In the following, we will omit the explicit distribution of the random variable in the subscript of $\mathbb{E}$ and replace expressions like $\mathbb{E}_{(Z,C)\sim \mathcal{P}}$ with $\mathbb{E}_{Z,C}$. The operator $\mathcal{A}_{Z,C}[\cdot]$ can be viewed as a mapping from $\mathcal{G}$ to $\mathcal{F}$ in the following way: by noting that $\langle \psi[C], g\rangle_{\mathcal{G}} = g(C)$ for any $g\in \mathcal{G}$, we can define $\mathcal{A}_{Z,C}[g]$ as
$$
\mathcal{A}_{Z,C}[g]\coloneqq \mathbb{E}_{Z,C}[g(C)\cdot \phi[Z]] - \mathbb{E}_C[g(C)]\cdot \mathbb{E}_Z[\phi[Z]].
$$
For any functions $f\in \mathcal{F}$ and $g\in \mathcal{G}$, the cross-covariace operator has the following property:
$$
\langle f, \mathcal{A}_{Z,C}[g] \rangle_{\mathcal{F}} = \mathbb{E}_{Z,C}[f(Z)\cdot g(C)] - \mathbb{E}_Z[f(Z)]\cdot \mathbb{E}_C[g(C)],
$$
which exactly corresponds to the covariance between $f(Z)$ and $g(C)$. The \textbf{conditional embedding operator} $\mathcal{U}_{Z|C}$ is a mapping from $\mathcal{G}$ to $\mathcal{F}$ such that, for any $c\in \mathcal{C}$, the follow equation holds:
\begin{equation}\label{eq:cond_embed_RKHS}
    \mathcal{U}_{Z|C}[\psi(c)] = \mu_{\mathcal{F}}[\mathcal{P}_{z|c}].
\end{equation}
In other words, $\mathcal{U}_{Z|C}$ maps the feature map $\psi(c)$ to the kernel mean embedding of the conditional distribution $Z|C=c$. Following \cite{song2009hilbert}, if the cross-covariance operator $\mathcal{A}_{C,C}$ is invertible, by defining
$$
\mathcal{U}_{Z|C}\coloneqq \mathcal{A}_{Z,C} \mathcal{A}_{C,C}^{-1}
$$
the equation \eqref{eq:cond_embed_RKHS} is satisfied. To see this, it is sufficient to show that $\langle f, \mathcal{A}_{Z,C} \mathcal{A}_{C,C}^{-1}[\psi[c]] \rangle_{\mathcal{F}} = \langle f, \mu_{\mathcal{F}}[\mathcal{P}_{z|c}] \rangle_{\mathcal{F}}$ for all $f\in \mathcal{F}$, and this holds following the derivations below,
$$
\begin{aligned}
    \langle f, \mu_{\mathcal{F}}[\mathcal{P}_{z|c}] \rangle_{\mathcal{F}} &= \mathbb{E}_{z|c}[f(Z)|c]\\
    &= \langle \mathbb{E}_{z|c}[f(Z)|C], \psi[c] \rangle_{\mathcal{G}}\\
    &= \langle \mathcal{A}_{C,C}[\mathbb{E}_{z|c}[f(Z)|C]], \mathcal{A}_{C,C}^{-1}[\psi[c]] \rangle_{\mathcal{G}}\\
    &= \langle \mathcal{A}_{C,Z}[f], \mathcal{A}_{C,C}^{-1}[\psi[c]]\rangle_{\mathcal{G}}\\
    &= \langle f, \mathcal{A}_{Z,C}\mathcal{A}_{C,C}^{-1}[\psi[c]]\rangle_{\mathcal{F}}.
\end{aligned}
$$
Equation \eqref{eq:cond_embed_RKHS} directly implies the following property, which serves as the key step of the KMM procedure in the density ratio estimation method of \cite{zhang2013domain}:
$$
\mathcal{U}_{Z|C}[\mu_{\mathcal{G}}[\mathcal{P}_c]] = \mu_{\mathcal{F}}[\mathcal{P}_z].
$$
\subsection{Empirical Estimations}\label{appdix:empirical_estimate}
In this section, we briefly outline how to estimate the quantities above using i.i.d. samples ${(z_i, c_i)}_{i=1}^N$ drawn from $\mathcal{P}$. By Mercer's theorem, the feature map $\phi[z]$ can be represented as a column vector in a (possibly infinite-dimensional) Hilbert space. We use $\phi[z] \psi[c]^\top$ to denote the outer product $\phi[z] \otimes \psi[c]$. Define the matrices $\mathbf{\Phi} \coloneqq (\phi[z_1], \dots, \phi[z_N])$ and $\mathbf{\Psi} \coloneqq (\psi[c_1], \dots, \psi[c_N])$. Further define $\mathbf{K}, \mathbf{H}\in \mathbb{R}^{N\times N}$ with $K_{i,j}=\kappa(z_i, z_j)$ and $H_{i,j}=\eta(c_i, c_j)$. The empirical estimators are indicated by adding a hat symbol $\hat{\cdot}$ to the original quantities, with their explicit expressions shown below:
$$
\begin{aligned}
    \mu_{\mathcal{F}}[\mathcal{P}_z]\approx \hat{\mu}_{\mathcal{F}} &\coloneqq \dfrac{1}{N}\sum_{i=1}^N \phi(z_i), \\
    \mu_{\mathcal{G}}[\mathcal{P}_c]\approx \hat{\mu}_{\mathcal{G}} &\coloneqq \dfrac{1}{N}\sum_{i=1}^N \psi(c_i), \\
    \mathcal{A}_{Z, C}\approx \hat{\mathcal{A}}_{Z,C} & \coloneqq \dfrac{1}{N}(\mathbf{\Phi} - \hat{\mu}_{\mathcal{F}} \mathbf{1}^\top ) (\mathbf{\Psi} - \hat{\mu}_{\mathcal{G}} \mathbf{1}^\top )^\top\\
    &= \dfrac{1}{N} \mathbf{\Phi} \left(I - \frac{1}{N} \mathbf{1} \mathbf{1}^\top\right) \mathbf{\Psi}^\top,\\
    \mathcal{A}_{C, C}^{-1}\approx \hat{\mathcal{A}}_{C, C}^{-1} &\coloneqq N\cdot \mathbf{\Psi} \mathbf{H}^{-1} \left(I - \frac{1}{N} \mathbf{1} \mathbf{1}^\top\right)^{-1} \mathbf{H}^{-1}\mathbf{\Psi}^\top, \\
    \mathcal{U}_{Z|C}\approx \hat{\mathcal{U}}_{Z|C} &\coloneqq \mathbf{\Phi} \mathbf{H}^{-1}\mathbf{\Psi}^\top.
\end{aligned}
$$

\section{Illustrative Example}\label{sec:illustrative_example}
Consider the illustrative example:
\begin{align}
    \min_x \ & \ \text{VaR}_{\alpha}^{\mathcal{Q}}(c\cdot x|z)  \ \ \text{s.t.}\ -1\le x\le 1, \label{LP:toy_appdix}
\end{align}
where $\alpha\in (0.5, 1)$ and $c=z +\epsilon$. In the training distribution $\mathcal{P}$, the $z$ and $\epsilon$ are independent and respectively follow $\mathcal{N}(0, \sigma_1^2)$ and $\mathcal{N}(0, \sigma_2^2)$. We list the explicit expressions for the following probability distributions:
$$
\mathcal{P}_z= \mathcal{N}(0, \sigma_1^2), \quad \mathcal{P}_c= \mathcal{N}(0, \sigma_1^2 + \sigma_2^2),
$$
$$
\mathcal{P}_{c|z}= \mathcal{N}(z, \sigma_2^2), \quad \mathcal{P}_{z|c}= \mathcal{N}\left(\frac{\sigma_1^2\cdot c}{\sigma_1^2 + \sigma_2^2}, \frac{\sigma_1^2\sigma_2^2}{\sigma_1^2 + \sigma_2^2}\right),
$$
$$
(c,z)\sim\mathcal{P}=\mathcal{N}\left(
\begin{pmatrix}
0\\
0
\end{pmatrix}
, \begin{pmatrix}
    \sigma_1^2 + \sigma_2^2 & \sigma_1^2\\
    \sigma_1^2 & \sigma_1^2 
\end{pmatrix}
\right).
$$
We describe the explicit form of $\mathcal{Q}$ under the following two distribution shift scenarios:
\begin{itemize}
    \item Covariate shift: let $s$ represent the extent of covariate shift. The distribution of $z$ is shifted from $\mathcal{P}_z = \mathcal{N}(0, \sigma_1^2)$ to $\mathcal{Q}_z = \mathcal{N}(s, \sigma_1^2)$, while the conditional distribution of $c|z$ remains to be $\mathcal{N}(z, \sigma_2^2)$. The joint distribution of $(c,z)$ is shifted to $\mathcal{Q} =\mathcal{N}\left(
 \begin{pmatrix}
s\\
s
\end{pmatrix}
, \begin{pmatrix}
    \sigma_1^2+\sigma_2^2 & \sigma_1^2\\
    \sigma_1^2 & \sigma_1^2
\end{pmatrix}
\right)$.
    \item Label shift: let $s$ represent the extent of label shift. The distribution of $c$ is shifted from $\mathcal{P}_c = \mathcal{N}(0, \sigma_1^2 + \sigma_2^2)$ to $\mathcal{Q}_c = \mathcal{N}(s, \sigma_1^2 + \sigma_2^2)$, while the conditional distribution $z|c$ remains to be $\mathcal{N}\left(\frac{\sigma_1^2\cdot c}{\sigma_1^2 + \sigma_2^2}, \frac{\sigma_1^2\sigma_2^2}{\sigma_1^2 + \sigma_2^2}\right)$. It follows that the distribution of $z$ in $\mathcal{Q}$ is shifted to $\mathcal{N}\left(\frac{\sigma_1^2\cdot s}{\sigma_1^2 + \sigma_2^2}, \sigma_1^2\right)$, and the distribution of $\epsilon$ is shifted to $\mathcal{N}\left(\frac{\sigma_2^2\cdot s}{\sigma_1^2 + \sigma_2^2}, \sigma_2^2\right)$. The joint distribution of $(c,z)$ is shifted to $ \mathcal{Q} = \mathcal{N}\left(
\begin{pmatrix}
s \\
\frac{\sigma_1^2\cdot s}{\sigma_1^2 + \sigma_2^2}
\end{pmatrix}
, \begin{pmatrix}
    \sigma_1^2 + \sigma_2^2 & \sigma_1^2\\
    \sigma_1^2 & \sigma_1^2 
\end{pmatrix}
\right)
$.
\end{itemize}
The solution to \eqref{LP:toy_appdix} admits a simple form under the idealized setting, as given in \eqref{eq:our_oracle_solution} and restated below:
\begin{equation}
    x^* = \left\{
\begin{aligned}
    &1, & \text{VaR}^{\mathcal{Q}}_\alpha(c|z) \leq 0,\\
    &-1, & \text{VaR}^{\mathcal{Q}}_{1-\alpha}(c|z) \geq 0,\\
    &0, & \text{otherwise},
\end{aligned}
\right.
\end{equation}
The probability of yielding an over-conservative solution is
\begin{equation}\label{eq:toy_solution_apx}
\begin{aligned}
    \mathbb{P}(x^* =0) &= \mathbb{P}_{z\sim \mathcal{Q}_z}(\text{VaR}_{1-\alpha}(c|z)\leq 0\leq \text{VaR}_\alpha(c|z)).
\end{aligned}
\end{equation}
Specifically, in the covariate shift setting, the probability equals
$$
\mathbb{P}(x^* =0) = \Phi\left(\frac{\sigma_2}{\sigma_1}\Phi^{-1}(\alpha)-\frac{s}{\sigma_1}\right) - \Phi\left(-\frac{\sigma_2}{\sigma_1}\Phi^{-1}(\alpha)-\frac{s}{\sigma_1}\right),
$$
and in the label shift setting, the probability also equals
$$
\mathbb{P}(x^* =0) = \Phi\left(\frac{\sigma_2}{\sigma_1}\Phi^{-1}(\alpha)-\frac{s}{\sigma_1}\right) - \Phi\left(-\frac{\sigma_2}{\sigma_1}\Phi^{-1}(\alpha)-\frac{s}{\sigma_1}\right).
$$

\paragraph{Worst-case Approach} Alternative to our density ratio-based approach is the worst-case approach.  Construct the worst-case ball: 
$$
\mathcal{B}(R)\coloneqq \left\{\mathcal{N}\left(
r
, \begin{pmatrix}
    \sigma_1^2  + \sigma_2^2 & \sigma_1^2\\
    \sigma_1^2 & \sigma_1^2
\end{pmatrix}
\right)\ : \ \|r\| \leq R\right\}.
$$
Following the discussions in the main context, the objective of the worst-case approach is given by \eqref{LP:toy_DRO} and restated below:
$$
    \min_x \max_{\mathcal{Q}\in \mathcal{B}(R)} \  \ \text{VaR}^{\mathcal{Q}}_{\alpha}(c\cdot x|z)  \ \ \text{s.t.}\ -1\le x\le 1.
$$
The probability of yielding an over-conservative solution is 
$$
\mathbb{P}(x^* =0) = \mathbb{P}_{z\sim \mathcal{Q}_z}(\underline{\text{VaR}}_{1-\alpha}(c|z)\leq 0\leq \overline{\text{VaR}}_\alpha(c|z)).
$$
In the covariate shift setting, we have $R = \sqrt{2} s$ in order that $\mathcal{B}(R)$ covers $\mathcal{Q}$. In the label shift setting, $R=\sqrt{1+\frac{\sigma_1^4}{(\sigma_1^2 + \sigma_2^2)^2}}\cdot  s$. The explicit form of the conservative probability is 
$$
\begin{aligned}
    \mathbb{P}(x^* =0) &=\Phi\left(\frac{\sigma_2}{\sigma_1}\Phi^{-1}(\alpha)+\frac{\sqrt{2}\cdot s}{\sigma_1}-\frac{s}{\sigma_1}\right) \\
    &\phantom{=} - \Phi\left(-\frac{\sigma_2}{\sigma_1}\Phi^{-1}(\alpha)-\frac{\sqrt{2}\cdot s}{\sigma_1}-\frac{s}{\sigma_1}\right)
\end{aligned}
$$
for the covariate shift setting and 
$$
\begin{aligned}
    \mathbb{P}(x^* =0) &=\Phi\left(\frac{\sigma_2}{\sigma_1}\Phi^{-1}(\alpha)+\sqrt{\frac{1}{\sigma_1^2}+\frac{\sigma_1^2}{(\sigma_1^2 + \sigma_2^2)^2}}\cdot  s -\frac{\sigma_1\cdot s}{\sigma_1^2 + \sigma_2^2}\right) \\
    &\phantom{=} - \Phi\left(-\frac{\sigma_2}{\sigma_1}\Phi^{-1}(\alpha)-\sqrt{\frac{1}{\sigma_1^2}+\frac{\sigma_1^2}{(\sigma_1^2 + \sigma_2^2)^2}}\cdot  s-\frac{\sigma_1\cdot s}{\sigma_1^2 +\sigma_2^2}\right)
\end{aligned}
$$
for the label shift setting.

Figure \ref{fig:toy_vs_OOD} is plotted by setting $\sigma_1 = \sigma_2 = 1$.

\section{More Experiments and Experiment Details}\label{sec:Experi}
We consider several implementations of the three components: $\hat{f}$, $\hat{h}$ and $\hat{w}$ for our Algorithm \ref{alg:BUQ-OOD}. For the expectation predictor $\hat{f}$ which predicts $\mathbb{E}[c|z]$ with $z$, we consider using Lasso regression, random forest, and neural network. On different training datasets, the regularization parameter $\lambda$ of the Lasso regression is selected as the optimal parameter in range $[0,4]$; for the random forest searches the best number of trees in range $[100, 1000]$ and the best depth in range $[10, 80]$; the neural network has a single middle layer with $16$ neurons, with the learning rate set to $0.01$ and the training epochs set to $500$. For the quantile predictor $\hat{h}$, we implement linear quantile regression, gradient boosting regression, and neural network, all trained by minimizing the quantile loss. For the density estimator, we mainly consider estimators for the covariate shift setting, including the trivial estimator, the kernel mean matching method introduced in \cite{gretton2008covariate}, and the probabilistic classification method introduced in \cite{bickel2009discriminative}. The computational complexity of the kernel mean matching method scales at least quadratically with the data size, so we only sample $200$ data from the training distribution and $200$ data from the test distribution to run the algorithm. 

The performance metrics we consider include the coverage rate and the $\alpha$-quantile of the objective $c^\top x$. Specifically, we evaluate the marginal and conditional coverage rates of the uncertainty sets generated from different implementations of our algorithm, and how different they are from the target coverage rate. The $\alpha$-quantile of the objective, denoted by $\mathrm{VaR}_{\alpha}^{\mathcal{Q}}(c^\top x|z)$, is evaluated by generating $100$ samples of $c$ from the underlying conditional distribution $\mathcal{Q}_{c|z}$ and then calculate the empirical $\alpha$-quantile of $c^\top x$.

\subsection{Additional Experiments on the Simple Example}
This section presents additional experimental results for the simple optimization problem discussed in Section \ref{sec:experiment}. In Figure \ref{fig:additional_simple}, we experiment with different choices of $\hat{f}$ and $\hat{h}$, and different covariate dimensions $d$. Specifically, the implementations of $\hat{f}$ and $\hat{h}$ follow the setup of Figure \ref{fig:vary_f_h}, and the $d$'s we consider include $d=2, 4, 8$. Both the mean squared error and the quantile loss are evaluated on the test dataset, with the $c$ values known. The results indicate that the algorithm tends to perform better, as reflected by a lower conservative probability, when the prediction models $\hat{f}$ and $\hat{h}$ have lower prediction errors. Further, for both $\hat{f}$ and $\hat{h}$, neural networks (NN) generally achieve strong predictive performance. 

\begin{figure}[ht!]
    \centering
    \begin{minipage}{0.5\columnwidth}
        \centering
        \includegraphics[width=\textwidth]{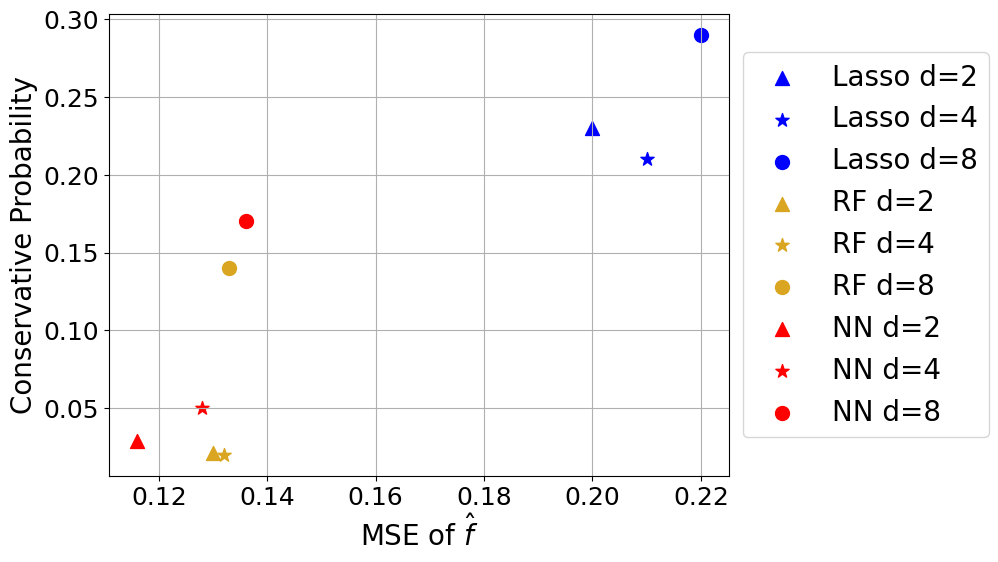} 
    \end{minipage}\hfill
    \begin{minipage}{0.5\columnwidth}
        \centering
        \includegraphics[width=\textwidth]{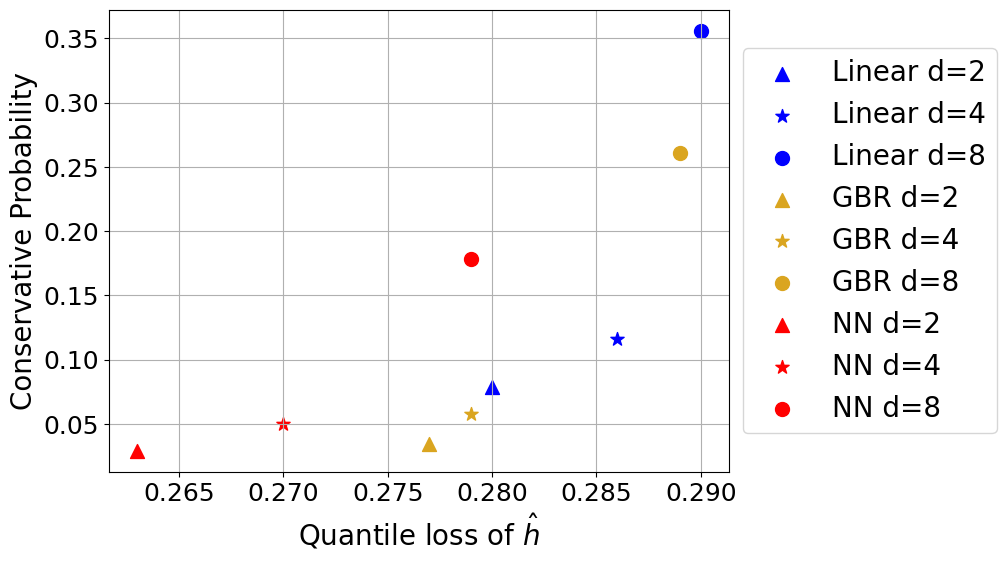} 
    \end{minipage}
    \caption{The quality of the predictors affects the performance of Algorithm \ref{alg:BUQ-OOD}. In the left figure, the x-axis represents the mean squared error (MSE) of $\hat{f}$ on the test dataset, while in the right figure, the x-axis represents the quantile loss of $\hat{h}$. For both metrics, lower values indicate higher predictor quality. The y-axis shows the probability of obtaining an over-conservative solution at $x^* = 0$, where a lower probability indicates better algorithm performance. Both figures illustrate that improved prediction models enhance the performance of the algorithm.}
    \label{fig:additional_simple}
\end{figure}

\subsection{Shortest Path Problem and Fractional Knapsack Problem}
We test our algorithm on two practical LP settings. The risk level $\alpha$ is fixed to $0.8$ and we consider the covariate shift of the test distribution.

The \textit{shortest path problem} seeks a path between two vertices in a weighted graph such that the total accumulated cost along the path's edges is minimized. Specifically, we consider the shortest path problem on a $5\times 5$ grid with $25$ nodes and $40$ edges. The objective is to find the shortest path from the top-left node to the bottom-right node. For $i=0,\ldots, 4$ and $j = 0,\ldots, 4$, let $(i,j)$ denote the node on the $i$-th row and $j$-th column, with the top-left node located at the coordinate $(0,0)$ and the bottom-right node at the coordinate $(5,5)$. The tuple $((i,j), (i', j'))_e$ denotes the edge between node $(i,j)$ and $(i', j')$. Let $V$ denote the set of all nodes and $E$ denote all edges. For $((i,j), (i', j'))_e\in E$, use $c_{(i,j), (i', j')}$ to represent the cost of the edge $((i,j), (i', j'))_e$ (for the undirected graph that we consider, set $c_{(i,j), (i', j')} = c_{(i',j'), (i, j)}$). The decision variables are $x_{(i,j), (i', j')}$ for all $((i,j), (i', j'))_e\in E$, where $x_{(i,j), (i', j')}=1$ denotes a directed path segment from the node $(i,j)$ to the node $(i', j')$. The shortest path problem then has the following risk-sensitive LP formulation:
\begin{equation}
    \begin{aligned}
    \min_{x} &\ \text{VaR}_\alpha \left(\sum_{((i,j), (i', j'))_e \in E} c_{(i,j), (i', j')}\cdot x_{(i,j), (i', j')}\right)\\
        \text{s.t.}\ & \ \sum_{(i', j'):((i,j), (i', j'))_e \in E}\left(x_{(i,j), (i', j')} - x_{(i',j'), (i, j)}\right) = \left\{
            \begin{aligned}
                1 & \quad (i,j) = (0,0)\\
                -1 & \quad (i,j) = (4,4)\\
                0 & \quad \text{otherwise}\\
            \end{aligned}
        \right.
\end{aligned}\label{eq:shortest_path_LP}
\end{equation}
Extending to the OOD robust formulation, we assume that the dynamic of the cost vector $c\in\mathbb{R}^{40}$ is controlled by the covariate $z\in \mathbb{R}^d$ (we fix $d=10$ in this experiment) through
$$
\mathcal{P}_{c|z} \sim \left(\left(\frac{1}{\sqrt{d}}\Theta z+3\right)^{5}+1\right)\circ \epsilon,
$$
where $\Theta\in\mathbb{R}^{40\times d}$ is a 0-1 matrix with each entry generated independently from a $\text{Bernoulli}(0.5)$ distribution. The $\Theta$ matrix is fixed once it is generated. The $\circ$ symbol denotes an elementwise multiplication. Each element of the random vector $\epsilon\in \mathbb{R}^{40}$ is generated independently from $\text{Uniform}(\frac{3}{4},\frac{5}{4})$. In the training data, the covariate $z$ is generated from $\mathcal{P}_z=\mathcal{N}(0, \mathbf{I}_d)$, and in the test data $z$ is generated from $\mathcal{Q}_z=\mathcal{N}(\mathbf{1}_d, \mathbf{I}_d)$. The objective of the OOD formulation replaces the $\text{VaR}_\alpha(\cdot)$ part in \eqref{eq:shortest_path_LP} with $\text{VaR}_\alpha^{\mathcal{Q}}(\cdot)$.

The \textit{fractional knapsack problem} models the case where customers select items that maximize their utility, under a budget constraint. We consider a simple setting with $20$ items. The decision variables $x=(x_1,\ldots, x_{20})$ denote the fractions (in range $[0, 1]$) of items to purchase, $c\in \mathbb{R}^{20}$ denote the item utilities, $p\in \mathbb{R}^{20}$ denote the price of items, and $B>0$ the total budget. The fractional knapsack problem has the following risk-sensitive LP formulation:
\begin{equation}
    \begin{aligned}
        \min_{x} &\ \text{VaR}_\alpha(- c^\top x)\\
        \text{s.t.}\ & \ p^\top x\leq B\\
        & x_i \in [0, 1], & i=1,\ldots, 20.
    \end{aligned}
\end{equation}
The OOD robust formulation considers covariate $z\in \mathbb{R}^d$ (with $d=10$) and assumes the conditional distribution of $c|z$ to be
$$
\mathcal{P}_{c|z}\sim (\Theta z)^2\circ \epsilon,
$$
where $\Theta\in \mathbb{R}^{20\times d}$ is a 0-1 matrix with each entry generated independently from a $\text{Bernoulli}(0.5)$ distribution.  Each element of $\epsilon\in \mathbb{R}^{20}$ is generated independently from $\text{Uniform}(\frac{4}{5},\frac{6}{5})$. The training distribution of the covariate $z$ is $\mathcal{P}_z=\mathcal{N}(0, \mathbf{I}_d)$ and the test distribution is $\mathcal{Q}_z=\mathcal{N}(\mathbf{1}_d, \mathbf{I}_d)$. 

In Figure \ref{fig:practical}, we compare our algorithm against multiple benchmarks on the two LP settings above. As there is no existing algorithm that handles covariate shift in the risk-sensitive setting, we just implement the existing benchmark methods and evaluate them in the test environment. The ``Ellipsoid'' method ignores the contextual information and calibrates the ellipsoid to achieve an empirical coverage rate of $\alpha$ on the training samples. The ``DCC'' and ``IDCC'' algorithms are proposed in \cite{chenreddy2022data}. The ``$k$NN'' algorithm is a conditional robust optimization method proposed in \citep{ohmori2021predictive}. The ``Ours-Trivial'' and ``Ours'' both implement our Algorithm \ref{alg:BUQ-OOD} (with $\hat{f}$ and $\hat{h}$ being neural networks).  ``Ours-Trivial'' sets a trivial density ratio estimator $\hat{w} \equiv 1$, while ``Ours'' uses the probabilistic classification method (which is shown to enjoy the best performance based on the previous experiments) to estimate $\hat{w}$. The result demonstrates that our algorithm generally outperforms the rest benchmarks, and leveraging the density ratio information further improves the performance on the test dataset.

\begin{figure}[ht!]
    \centering
    \begin{minipage}{0.5\columnwidth}
        \centering
        \includegraphics[width=\textwidth]{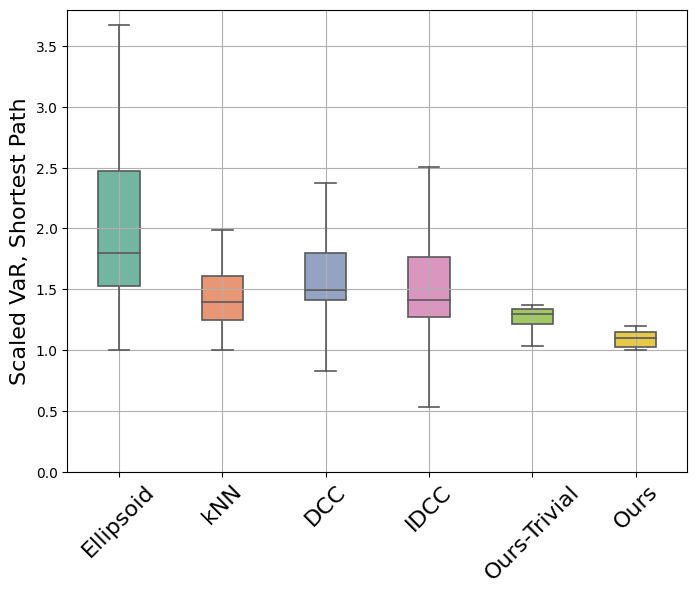} 
    \end{minipage}\hfill
    \begin{minipage}{0.5\columnwidth}
        \centering
        \includegraphics[width=\textwidth]{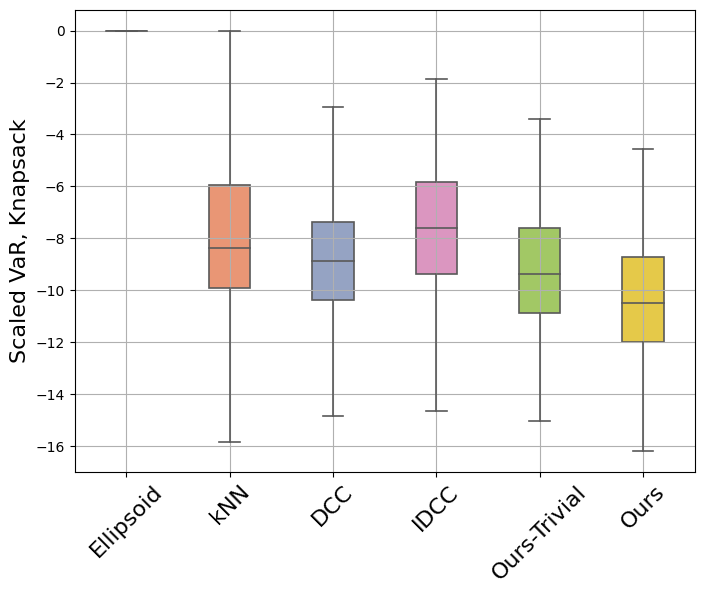} 
    \end{minipage}
    \caption{The scaled VaR for the objective in the shortest path and fractional knapsack problems. The x-axis lists the algorithms we test. The ``Ellipsoid'', ``kNN'', ``DCC'' and ``IDCC'' algorithms are introduced in the experiment descriptions. ``Ours-Trivial'' implements our algorithm but uses a trivial density ratio estimator $\hat{w}\equiv 1$. ``Ours'' implements our algorithm and uses the probabilistic classification method to estimate $\hat{w}$. Amongst the benchmark algorithms, our proposed Algorithm \ref{alg:BUQ-OOD} generally achieves the lowest VaR.}
    \label{fig:practical}
\end{figure}

\section{Proof of Theorems}
\subsection{Proof of Theorem \ref{thm:cvg_alpha}}
Let $N\coloneqq |\mathcal{D}_{2}|$ and define $x_i\coloneqq (c_i,z_i)$ for $i=1,\ldots, N$. The final $\eta$ term produced from Algorithm \ref{alg:BUQ-OOD} is a function of $(x_1,\ldots, x_N)$, which we define below as the $\hat{\eta}(\cdot)$ function
$$
\hat{\eta}(x_1, \ldots, x_N)\coloneqq \min\left\{\eta\geq 0 \ : \ \sum_{i=1}^N w(x_i)\cdot \mathbbm{1}\{|\hat{f}(z_i)-c_i|\leq \eta\cdot \hat{h}(z_i)\}\geq \alpha\cdot \sum_{i=1}^N w(x_i)\right\}.
$$
For each $x_i$, let $\Tilde{\eta}(x_i)$ denote the minimum $\eta$ such that $[\hat{f}(z_i) -\eta\cdot \hat{h}(z_i), \hat{f}(z_i) +\eta\cdot \hat{h}(z_i)]$ covers $c_i$. The $\Tilde{\eta}(\cdot)$ function is formally defined below
$$
\Tilde{\eta}(x_i)\coloneqq \min\left\{\eta\geq 0 \ : \ 
 |\hat{f}(z_i)-c_i|\leq \eta\cdot \hat{h}(z_i)\right\}.
$$
Under this new notation, the form of $\hat{\eta}(\cdot)$ can be formulated as
$$
\hat{\eta}(x_1,\ldots, x_N) = \min\left\{\eta\geq 0 \ : \ \sum_{i=1}^N w(x_i)\cdot \mathbbm{1}\{\eta \geq \Tilde{\eta}(x_i)\}\geq \alpha\cdot \sum_{i=1}^N w(x_i) \right\}.
$$
Further, if the density ratio estimate in Algorithm \ref{alg:BUQ-OOD} is perfect, then the event $\{c_{\mathrm{new}} \in \mathcal{U}_\alpha(z_{\mathrm{new}})\}$ is equivent to $\{\Tilde{\eta}(x_{\mathrm{new}})\leq \hat{\eta}(x_1,\ldots, x_N)\}$. We restate Theorem \ref{thm:cvg_alpha} below and provide a proof.
\begin{thm} Under Assumption \ref{assump:u_l_bound}, suppose the density ratio estimate is perfect, i.e., $\hat{w}(c,z)=w(c,z)=q(c,z)/p(c,z)$, then the uncertainty set $\mathcal{U}_{\alpha}(z)$ generated by Algorithm \ref{alg:BUQ-OOD} satisfies the following coverage guarantee,
    \begin{equation*}
        \Big| \mathbb{P}\left(\Tilde{\eta}(x_{\mathrm{new}})\leq \hat{\eta}(x_1,\ldots, x_N)\right) -\alpha\Big|\leq \dfrac{1}{N+1}\cdot \dfrac{\bar{w}}{\underline{w}}
    \end{equation*}
    where the probability on the left-hand-side is with respect to $x_{\mathrm{new}}\sim \mathcal{Q}$ and $x_1,\ldots, x_N \sim \mathcal{P}$.
\end{thm}
\begin{proof}
    Since the distributions $\mathcal{P}$ and $\mathcal{Q}$ are continuous, almost surely the samples $x_1, \ldots, x_N, x_{\mathrm{new}}$ are mutually distinct. Let $\{\cdot\}$ denote an unordered set (e.g. $\{x_1, \ldots, x_N\}$ denotes an unordered set containing distinct elements $x_1, \ldots, x_N$). Then the following equation holds:
    \begin{equation}\label{eq:into_integration}
    \begin{aligned}
        &\mathbb{P}(\Tilde{\eta}(x_{\mathrm{new}})\leq \hat{\eta}(x_1,\ldots, x_N))\\
        &= \int_{\{a_1,\ldots, a_{N+1}\}}\mathbb{P}\bigg(\Tilde{\eta}(x_{\mathrm{new}})\leq \hat{\eta}(x_1,\ldots, x_N) \ \bigg| \ \{x_1,\ldots, x_n, x_{\mathrm{new}}\} = \{a_1,\ldots, a_{N+1}\}\bigg) \\
        & \phantom{=\int_{\{a_1,\ldots, a_{N+1}\}}}\cdot \mathbb{P}\bigg(\{x_1,\ldots, x_n, x_{\mathrm{new}}\} = \{a_1,\ldots, a_{N+1}\}\bigg).
    \end{aligned}
    \end{equation}
    The integration is over all possible sets of $N+1$ distinct elements, denoted by $\{a_1,\ldots, a_{N+1}\}$. The remaining part of the proof uniformly bounds the $\mathbb{P}\bigg(\Tilde{\eta}(x_{\mathrm{new}})\leq \hat{\eta}(x_1,\ldots, x_N) \ \bigg| \ \{x_1,\ldots, x_n, x_{\mathrm{new}}\} = \{a_1,\ldots, a_{N+1}\}\bigg)$ term for all sets $\{a_1,\ldots, a_{N+1}\}$.

    Given a set $\{a_1,\ldots, a_{N+1}\}$, let $E$ denote the event that $\{x_1,\ldots, x_n, x_{\mathrm{new}}\} = \{a_1,\ldots, a_{N+1}\}$, and let $E_i$ denote the event that $x_{\mathrm{new}} = a_i$ and $\{x_1, \ldots, x_N\} = \{a_1,\ldots, a_{N+1}\} \backslash \{a_i\}$. The $\mathbb{P}\bigg(\Tilde{\eta}(x_{\mathrm{new}})\leq \hat{\eta}(x_1,\ldots, x_N) \ \bigg| \ E\bigg)$ term can be decomposed by the following chain of equations:
    \begin{equation}\label{eq:starting_derivations}
        \begin{aligned}
        & \mathbb{P}\bigg(\Tilde{\eta}(x_{\mathrm{new}})\leq \hat{\eta}(x_1,\ldots, x_N) \ \bigg| \ E\bigg)\\
        & = \sum_{i=1}^{N+1} \mathbb{P}(E_i\mid E)\cdot \mathbb{P}\bigg(\Tilde{\eta}(x_{\mathrm{new}})\leq \hat{\eta}(x_1,\ldots, x_N) \ \bigg| \ E_i\bigg) \\
        & \stackrel{(1)}{=} \sum_{i=1}^{N+1} \mathbb{P}(E_i\mid E)\cdot \mathbbm{1}\bigg\{\Tilde{\eta}(x_{\mathrm{new}})\leq \hat{\eta}(x_1,\ldots, x_N) \ \bigg| \ E_i\bigg\}\\
        & \stackrel{(2)}{=} \sum_{i=1}^{N+1}\dfrac{w(a_i)}{\sum_{j=1}^{N+1}w(a_j)}\cdot \mathbbm{1}\bigg\{\Tilde{\eta}(x_{\mathrm{new}})\leq \hat{\eta}(x_1,\ldots, x_N) \ \bigg| \ E_i\bigg\},
    \end{aligned}
    \end{equation}
    where $(1)$ is from the fact that $\hat{\eta}(x_1, \ldots, x_N)$ is invariant to the permutation of $(x_1, \ldots, x_N)$, and $(2)$ is by that 
    $$
    \begin{aligned}
        \mathbb{P}(E_i) &= N!\cdot q(a_i)\cdot \prod_{j\neq i}p(a_j) \\
        &= N!\cdot w(a_i)\cdot \prod_{j=1}^{N+1}p(a_j),
    \end{aligned}
    $$
    and $\mathbb{P}(E) = \sum_{i=1}^{N+1}\mathbb{P}(E_i)$.

    We now characterize the terms  $\mathbbm{1}\bigg\{\Tilde{\eta}(x_{\mathrm{new}})\leq \hat{\eta}(x_1,\ldots, x_N) \ \bigg| \ E_i\bigg\}$. Without loss of generality, let $\Tilde{\eta}(a_1)<\Tilde{\eta}(a_2)\cdots <\Tilde{\eta}(a_{N+1})$, then 
    \begin{equation}\label{eq:indicator_equation}
        \begin{aligned}
        & \mathbbm{1}\bigg\{\Tilde{\eta}(x_{\mathrm{new}})\leq \hat{\eta}(x_1,\ldots, x_N) \ \bigg| \ E_i\bigg\} \\
        & = \mathbbm{1}\bigg\{\Tilde{\eta}(a_i)\leq \hat{\eta}(\{a_1,\ldots, a_{N+1}\} \backslash \{a_i\}) \bigg\}\\
        & = \mathbbm{1}\bigg\{\sum_{j=1}^{i-1}w(a_j) < \alpha\cdot \sum_{j\neq i} w(a_j)  \bigg\}.
    \end{aligned}
    \end{equation}
    The last line of \eqref{eq:indicator_equation} has the following bounds:
    \begin{equation}\label{eq:indicator_bound}
        \mathbbm{1}\bigg\{\sum_{j=1}^{i}w(a_j) < \alpha\cdot \sum_{j=1}^{N+1} w(a_j)  \bigg\} \leq \mathbbm{1}\bigg\{\sum_{j=1}^{i-1}w(a_j) < \alpha\cdot \sum_{j\neq i} w(a_j)  \bigg\} \leq \mathbbm{1}\bigg\{\sum_{j=1}^{i-1}w(a_j) < \alpha\cdot \sum_{j=1}^{N+1} w(a_j)  \bigg\}.
    \end{equation}
    Combining \eqref{eq:starting_derivations}, \eqref{eq:indicator_equation} and \eqref{eq:indicator_bound}, the $\mathbb{P}\bigg(\Tilde{\eta}(x_{\mathrm{new}})\leq \hat{\eta}(x_1,\ldots, x_N) \ \bigg| \ \{x_1,\ldots, x_n, x_{\mathrm{new}}\} = \{a_1,\ldots, a_{N+1}\}\bigg)$ term is uniformly bounded:
    $$
    \begin{aligned}
        & \mathbb{P}\bigg(\Tilde{\eta}(x_{\mathrm{new}})\leq \hat{\eta}(x_1,\ldots, x_N) \ \bigg| \ E\bigg)\\
        & \geq \sum_{i=1}^{N+1}\dfrac{w(a_i)}{\sum_{j=1}^{N+1}w(a_j)}\cdot \mathbbm{1}\bigg\{\sum_{j=1}^{i}w(a_j) < \alpha\cdot \sum_{j=1}^{N+1} w(a_j)\bigg\}\\
        & \geq \alpha - \dfrac{\bar{w}}{\sum_{j=1}^{N+1} w(a_j)} \geq \alpha - \dfrac{1}{N+1}\cdot \dfrac{\bar{w}}{\underline{w}}.
    \end{aligned}
    $$
    $$
    \begin{aligned}
        & \mathbb{P}\bigg(\Tilde{\eta}(x_{\mathrm{new}})\leq \hat{\eta}(x_1,\ldots, x_N) \ \bigg| \ E\bigg)\\
        & \leq \sum_{i=1}^{N+1}\dfrac{w(a_i)}{\sum_{j=1}^{N+1}w(a_j)}\cdot \mathbbm{1}\bigg\{\sum_{j=1}^{i-1}w(a_j) < \alpha\cdot \sum_{j=1}^{N+1} w(a_j)\bigg\}\\
        & \leq \alpha + \dfrac{\bar{w}}{\sum_{j=1}^{N+1} w(a_j)} \leq \alpha + \dfrac{1}{N+1}\cdot \dfrac{\bar{w}}{\underline{w}}.
    \end{aligned}
    $$
    Plugging the bounds above into \eqref{eq:into_integration} finishes the proof.

\end{proof}

\subsection{Proof of Corollary \ref{coro:guarantee}}
By defining the estimated test distribution $\hat{\mathcal{Q}}$ as a distribution with the following density function:
$$
    \hat{q}(c,z) = \dfrac{\hat{w}(c,z)\cdot p(c,z)}{\int_{(c,z)}\hat{w}(c,z)\cdot p(c,z) \mathrm{d}z\mathrm{d}c},
$$
we could directly apply Theorem \ref{thm:cvg_alpha} to see that 
\begin{equation}\label{eq:thm_1_apx}
    \Big| \mathbb{P}_{\hat{\mathcal{Q}}}\left(c_{\mathrm{new}} \in \mathcal{U}_\alpha(z_{\mathrm{new}})\right) -\alpha\Big|\leq \dfrac{1}{|\mathcal{D}_2|+1}\cdot \dfrac{\bar{w}}{\underline{w}},
\end{equation}
where the probability $\mathbb{P}_{\hat{\mathcal{Q}}}$ is with respect to $(c_{\text{new}},z_{\text{new}})\sim \hat{\mathcal{Q}}$ and $\mathcal{D}_2 \sim \mathcal{P}$. From the definition of the total variation distance, the following inequality holds:
\begin{equation}\label{eq:dtv_apx}
    \Big| \mathbb{P}_{\hat{\mathcal{Q}}}\left(c_{\mathrm{new}} \in \mathcal{U}_\alpha(z_{\mathrm{new}})\right) -\mathbb{P}\left(c_{\mathrm{new}} \in \mathcal{U}_\alpha(z_{\mathrm{new}})\right)\Big|\leq D_{\text{TV}}(\mathcal{Q}, \hat{\mathcal{Q}}),
\end{equation}
where $\mathbb{P}$ is with respect to $(c_{\text{new}},z_{\text{new}})\sim \mathcal{Q}$ and $\mathcal{D}_2 \sim \mathcal{P}$. Combining \eqref{eq:thm_1_apx} and \eqref{eq:dtv_apx} gives the result of Corollary \ref{coro:guarantee}, which we restate below:
\begin{equation*}
    \Big|\mathbb{P}\left(c_{\mathrm{new}} \in \mathcal{U}_\alpha(z_{\mathrm{new}})\right) -\alpha\Big| \leq \dfrac{1}{|\mathcal{D}_2|+1}\cdot \dfrac{\bar{w}}{\underline{w}} + D_{\mathrm{TV}}(\mathcal{Q}, \hat{\mathcal{Q}}).
\end{equation*}

\end{document}